\documentclass[aos,numbers,preprint]{imsart}

\usepackage{times}
\usepackage[margin=1in]{geometry}

\usepackage{amsfonts,amsthm,amsmath,amssymb,amstext,graphicx}
\usepackage{mathrsfs}

\RequirePackage[OT1]{fontenc}
\RequirePackage{hyperref}
\RequirePackage{hypernat}
\RequirePackage{cite}

\allowdisplaybreaks

\arxiv{math.ST/???}

\startlocaldefs

\numberwithin{equation}{section} 

\newtheorem{theorem}{Theorem}[section]
\newtheorem{lemma}{Lemma}[section]
\newtheorem{corollary}{Corollary}[section]
\newtheorem{proposition}{Proposition}[section]

\theoremstyle{definition} 

\newtheorem{remark}{Remark}[section]

\def\Var{\mathop{\rm Var}\nolimits}

\newcommand{\mb}[1]{\mathbf{#1}} 
\newcommand{\Cb}{{\mb{C}}}

\newcommand{\Xb}{{\mb{X}}}

\def\A{{\mb{C}}}
\def\C{{\mb{C}}}
\def\pb{{\mb{p}}}

\newcommand{\mbs}[1]{\boldsymbol{#1}} 

\newcommand{\teb}{{\mbs{\theta}}}

\newcommand{\mc}[1]{\mathcal{#1}} 
\newcommand{\Rc}{{\mc{R}}}
\newcommand{\Rca}{{\overline{\mc{R}}}}

\newcommand{\Nc}{{\mc{N}}}
\newcommand{\Fc}{{\mc{F}}}

\newcommand{\Ac}{{\mc{A}}}

\newcommand{\Bc}{{\mc{B}}}

\newcommand{\Lc}{{\mc{L}}}
\newcommand{\Pc}{{\mc{P}}}

\newcommand{\wtT}{\widetilde{T}}
\newcommand{\wtX}{\widetilde{X}}
\newcommand{\wtS}{\widetilde{S}}

\usepackage{mathrsfs}
 \newcommand{\PFA}{\mathsf{PFA}}
 \newcommand{\pfa}{\mathsf{PFA}}
\newcommand{\Pb}{{\mathsf{P}}} 
\newcommand{\EV}{{\mathsf{E}}} 
\newcommand{\Eb}{{\mathsf{E}}}
\newcommand{\Hyp}{{\mathsf{H}}} 

\newcommand{\mrm}[1]{\mathrm{#1}}
\newcommand{\drm}{{\mrm{d}}}

\newcommand{\mbb}[1]{\mathbb{#1}} 
\def\One{\mathchoice{\rm 1\mskip-4.2mu l}{\rm 1\mskip-4.2mu l}
{\rm 1\mskip-4.6mu l}{\rm 1\mskip-5.2mu l}}
\newcommand\Ind[1]{{\One_{\{#1\}}}} 

\newcommand{\Zbb}{\mbb{Z}} 

\newcommand{\class}{{\mbb{C}}(\alpha)}

\def\bbr{{\mathbb R}}

\newcommand{\xra}{\xrightarrow} 

\newcommand{\abs}[1]{\left\vert#1\right\vert}
\newcommand{\set}[1]{\left\{#1\right\}}

\newcommand{\brc}[1]{\left(#1\right)}
\newcommand{\brcs}[1]{\left[#1\right]}

\renewcommand{\le}{\leqslant} 
\renewcommand{\ge}{\geqslant}

\newcommand{\ignore}[1]{} 

\usepackage{color}

\endlocaldefs

\begin{document}

\begin{frontmatter}
\title{ASYMPTOTICALLY OPTIMAL QUICKEST CHANGE DETECTION IN MULTISTREAM DATA---PART 1: GENERAL STOCHASTIC MODELS} 
\runtitle{Quickest Change Detection in Multistream Data---Part 1}

\begin{aug}
\author{\large \fnms{Alexander G.} \snm{Tartakovsky}\ead[label=e1]{agt@phystech.edu}\thanksref{t1}}
\thankstext{t1}{A.G. Tartakovsky is a Head of the Space informatics Laboratory at the Moscow Institute of Physics and Technology, Russia
 and Vice President of AGT StatConsult, Los Angeles, California, USA. The work was supported in part by the Russian Federation Ministry of Science and Education Arctic program and the grant 18-19-00452
from the Russian Science Foundation at the Moscow Institute of Physics and Technology.}
\address{Moscow Institute of Physics and Technology\\
Laboratory of Space Informatics\\
Moscow, Russia\\
and\\
AGT StatConsult \\
Los Angeles, CA, USA\\
}

\affiliation{Moscow Institute of Physics and Technology and AGT StatConsult}

\runauthor{A.G. Tartakovsky}

\end{aug}

\begin{abstract}
Assume that there are multiple data streams (channels, sensors) and in each stream the process of interest produces generally dependent and non-identically distributed observations. 
When the process is in a normal mode (in-control), the (pre-change) distribution is known, but when the process becomes abnormal there is a parametric uncertainty, i.e., the post-change 
(out-of-control) distribution is known only partially up to a parameter.  Both the change point and the post-change parameter are unknown. Moreover, the change affects an unknown subset of streams, so that the number of affected streams and their location are unknown in advance. A good changepoint detection procedure should detect the change as soon as possible
after its occurrence while controlling for a risk of false alarms.   We consider a Bayesian setup with a given prior distribution of the change point and propose two sequential mixture-based
change detection rules, one mixes a Shiryaev-type statistic over both the unknown subset of affected streams and the unknown post-change parameter and another mixes a 
Shiryaev--Roberts-type statistic.  These rules generalize the mixture detection procedures studied by Tartakovsky (2018) in a single-stream case.  We provide sufficient conditions under which the proposed multistream change detection procedures are first-order asymptotically optimal with respect to moments of the delay to detection as the probability of false alarm approaches zero.
\end{abstract}



\end{frontmatter}



%
\section{Introduction} \label{sec:intro}

In most surveillance applications with unknown points of change, including the classical Statistical Process Control sphere,
the baseline (pre-change, in-control) distribution of observed data is known, but the post-change out-of-control distribution is not completely known. There are 
three conventional approaches in this case: (i)  to select a representative value of the post-change parameter and apply efficient detection procedures tuned to this value such as the 
Shiryaev procedure, the Shiryaev--Roberts procedure or CUSUM, (ii) to select a mixing measure over the parameter space and apply mixture-type procedures, 
(iii) to estimate the parameter and apply adaptive schemes. In the present article, we consider a more general case where the change occurs in multiple data streams and more 
general multi-stream double-mixture-type change detection procedures, assuming that the number and location of affected data streams are also unknown.

To be more specific, suppose there are $N$ data streams $\{X_i(n)\}_{n \ge 1}$, $i=1,\dots,N$, observed sequentially in time subject to a change at an unknown time $\nu\in\{0,1, 2, \dots\}$, 
so that $X_i(1),\dots,X_i(\nu)$ are generated by one stochastic model and $X_i(\nu+1),  X_i(\nu+2), \dots$ by another model when the change occurs in the $i$th stream. 
The change in distributions happens at a subset of streams 
$\Bc\in \{1,\dots,N\}$ of a size (cardinality) $1\le |\Bc| \le K \le N$, where $K$ is an assumed maximal number of streams that can be affected, which can be and often is substantially smaller than $N$. 
A sequential detection rule is a stopping time $T$ with respect to an observed sequence 
$\{\Xb(n)\}_{n\ge 1}$, $\Xb(n)=(X_1(n),\dots,X_N(n))$, i.e., $T$ is an integer-valued random variable, such that the event $\{T = n\}$
belongs to the sigma-algebra  $\Fc_{n}=\sigma(\Xb^n)$ generated by observations $\Xb(1),\dots,\Xb(n)$.  
A false alarm is raised when the detection is declared before the change occurs. 
We want to detect the change with as small a delay as possible while controlling for a risk of false alarms. 

To begin, assume for the sake of simplicity that the observations are independent across data streams, 
but have a fairly general stochastic structure within streams. So if we let $\Xb_i^{n}=(X_i(1),\dots,X_i(n))$ denote the sample of size $n$ in the $i$th stream and if
$\{f_{\theta_i,n}(X_i(n)|\Xb_i^{n-1})\}_{n\ge 1}$, $\theta_i\in\Theta_i$ is a parametric family of conditional densities of $X_i(n)$ 
given $\Xb_i^{n-1}$, then when $\nu=\infty$ (there is no change) the parameter $\theta_i$ is equal to the known value $\theta_{0,i}$, i.e., 
$f_{\theta_i,n}(X_i(n)|\Xb_i^{n-1})=f_{\theta_{i,0},n}(X_i({n}|\Xb_i^{n-1})$ for all 
$n \ge 1$ and when $\nu=k<\infty$, then $\theta_i=\theta_{i,1}\neq \theta_{i,0}$, i.e., $f_{\theta_i,n}(X_i(n)|\Xb_i^{n-1})=f_{\theta_{0,i},n}(X_i(n)|\Xb_i^{n-1})$ for $n \le k$ and 
$f_{\theta,n}(X_i(n)|\Xb_i^{n-1})=f_{\theta_{i,1},n}(X_{n}|\Xb^{n-1})$ for $n > k$. Not only the point of change $\nu$, but also the subset $\Bc$, its size $|\Bc|$, and the post-change parameters 
$\theta_{i,1}$ are unknown.

In the case where $f_{\theta_i,n}(X_i(n)|\Xb_i^{n-1})= f_{\theta_i}(X_i(n))$, i.e., when  the observations in the $i$th stream are independent and identically distributed (i.i.d.) according to a distribution
with density $f_{\theta_{0,i}}(X_i(n))$ in the pre-change mode and with density $ f_{\theta_{1,i}}(X_n)$ in the post-change mode this problem was considered in
\cite{Mei-B2010,felsokIEEEIT2016,Xie&Siegmund-AS13, Tartakovskyetal-SM06,TartakovskyIEEECDC05,TNB_book2014}. 
Specifically, in the case of a known post-change parameter and $K=1$ (i.e., when only one stream can be affected but it is unknown which one), Tartakovsky~\cite{TartakovskyIEEECDC05} 
proposed to use a multi-chart CUSUM procedure that raises an alarm when one of the partial CUSUM statistics exceeds a threshold. This procedure is very simple, 
but it is not optimal and performs poorly when many data streams are affected. To avoid this drawback, Mei~\cite{Mei-B2010} suggested a SUM-CUSUM rule based on the sum of CUSUM 
statistics in streams and evaluated its first-order performance, which shows that this detection scheme is first-order asymptotically minimax minimizing the maximal expected delay to detection
when the average run length (ARL) to false alarm approaches infinity. Fellouris and Sokolov~\cite{felsokIEEEIT2016} suggested more efficient generalized and mixture-based CUSUM rules 
that are second-order minimax.  Xie and Siegmund~\cite{Xie&Siegmund-AS13} considered a particular Gaussian model with an unknown post-change mean. 
They suggested a rule that combines mixture likelihood ratios that incorporate an assumption about the proportion of affected data streams with the generalized 
CUSUM statistics in streams and then add up the resulting local statistics. They also performed a detailed asymptotic analysis of the proposed
detection procedure in terms of the ALR to false alarm and the expected delay as well as MC simulations. Chan~\cite{Chan-AS2017} studied a version of the 
mixture likelihood ratio rule for detecting a change in the mean of the normal population assuming independence of data streams and established its asymptotic optimality  in a minimax setting as well as 
dependence of operating characteristics on the fraction of affected streams.

In the present paper, we consider a Bayesian problem with a general prior distribution of the change point and we generalize the results of Tartakovsky~\cite{TartakovskyIEEEIT2018} 
for a single data stream and a general stochastic model to multiple data streams with an unknown pattern, i.e., when the size and location of the affected streams are unknown. 
It is assumed that the observations can be dependent and non-identically distributed in data streams and even across the streams. We introduce two double-mixture detection procedures -- the first one
mixes the Shiryaev-type statistic over the distributions of the unknown pattern and unknown post-change parameter; the second one is the double-mixture Shiryaev--Roberts statistic. 
The resulting statistics are then compared to appropriate thresholds.  The main contribution of the present article (Part 1), as well as of the companion article (Part 2), is two-fold. In Part 1, 
we present a general theory for very general stochastic models, providing sufficient conditions under which the suggested detection procedures are first-order asymptotically optimal.  
In the companion article, we will consider the ``i.i.d.'' case, where data streams are mutually independent and also data
in each stream are independent, and we will provide higher-order asymptotic approximations to the operating characteristics -- the average detection delay and the probability of false alarm. 
We will also examine the accuracy of these approximations and compare the performance of several procedures by Monte Carlo simulations.

The remainder of the paper is organized as follows.  In Section~\ref{sec:Procedures}, we introduce notation and describe a general stochastic model and detection procedures. 
In Section~\ref{sec:Problem}, we formulate the asymptotic optimization problems and assumptions on the prior distribution of the change point and on the model. In Section~\ref{sec:LowerBounds},
we provide asymptotic lower bounds for moments of the detection delay in the class of detection procedures with the given weighted probability of false alarm, which are then used in
Section~\ref{sec:AODMS} and in Section~\ref{sec:AOIR} for establishing first-order asymptotic optimality property of the double-mixture detection rules with respect to moments of detection delay
as the probability of false alarm and cost of delay in change detection approach zero. Section~\ref{sec:AOMS} provides a connection with the problem where the post-change parameter is either 
known or pre-selected. In Section~\ref{sec:indstreams}, the results are specified in the case of mutually independent data streams, which was the basic assumption in all previous publications, 
but we still assume that the observations in streams are non-i.i.d. Section~\ref{sec:Ex} illustrates general results by examples that justify asymptotic optimality properties of proposed detection 
procedures. Section~\ref{sec:Remarks} concludes the paper with remarks and a short discussion.

\section{A multistream model and change detection procedures based on mixtures} \label{sec:Procedures}

\subsection{The general multistream model}\label{ssec:Model}

Consider the multistream scenario where the observations $\Xb=(X_{1}, \ldots, X_{N})$ are sequentially acquired in $N$ streams (sources, channels), i.e.,
in the $i$th stream one observes a sequence $X_i=\{X_{i}(n)\}_{n \ge1}$, where $i \in [N]:=\{1, \ldots,N\}$.   Let $\Pb_\infty$ denote the probability measure corresponding to the 
sequence of observations $\{\Xb_n\}_{n\ge 1}$  from all $N$ streams when there is never a change ($\nu=\infty$) in any of the components and,  for $k=0,1,\dots$ and $\Bc \subset [N]$, 
let $\Pb_{k,\Bc}$ denote the measure  corresponding to the sequence  $\{\Xb_n\}_{n\ge 1}$ when $\nu=k<\infty$ and the change occurs in a subset $\Bc$ of the set $\Pc$ 
(i.e., $X_i(\nu+1)$, $i\in \Bc$ is the first post-change observation). It is convenient to parametrize the post-change distribution $\Pb_{k,\eta}$ 
of  $\Xb=(X_{1}, \ldots, X_{N})$ by  an $N$-dimensional parameter vector, 
$\eta=(\eta_1, \ldots, \eta_{N})$, where each component $\eta_i$ takes values in the binary set $\{0,1\}$, $i \in [N]$.   
Let  $\Hyp_\infty$ denote the  hypothesis that there is no change, under which all components of $\eta$ are equal to 0. For any subset of components, $\Bc \subset [N]$, let $\Hyp_{k, \Bc}$ 
be the hypothesis according to which  only the  components of $\eta$ in $\Bc$ are non-zero after the change point $\nu=k$,  i.e., 
\begin{align} \label{eq:setup:param_change}
    \begin{split}
\eta_i&=0, \qquad i \in  [N] \qquad \text{under} \quad \Hyp_\infty \\
        \eta_i &=
        \begin{cases}
            1 , & \quad i \in \Bc, \\
            0 , &  \quad i \notin \Bc,
        \end{cases}
        \qquad \text{under} \quad \Hyp_{k, \Bc}.
    \end{split}
\end{align}
The set $\Pc$ is a  class of subsets of $[N]$ that  incorporates available prior information regarding the subset of non-zero components of $\eta$.    For example,  
when it is known that \textit{exactly} $K$ streams can be affected after the change occurs,  then $\Pc=\widetilde{\Pc}_K$, and when it is known  that \textit{at most} $K$ 
channels can be affected, then  $\Pc=\Pc_K$, where
\begin{align} \label{upperlowerclass2}
\begin{split}
\widetilde{\Pc}_K & = \{\Bc \subset [N] : |\Bc|=K\} , \\
\Pc_K &= \{\Bc \subset [N] : 1 \le |\Bc| \le K\}.  
\end{split}
\end{align}
Hereafter we denote by $|\Bc|$ the size of a subset $\Bc$, i.e., the number of non-zero components under $\Hyp_{k, \Bc}$ and $|\Pc|$ denotes the size of class $\Pc$, i.e., the number of possible alternatives 
in $\Pc$. Note that $|\Pc|$ takes maximum value when there is no prior information regarding the subset of affected components of $\eta$, i.e., when  $\Pc=\Pc_{N}$, in which case $|\Pc|=2^{N}-1$.

We will write $\Xb_i^{n}=(X_i(1),\dots,X_i(n))$ for the concatenation of the first $n$ observations from the $i$th data stream and $\Xb^{n}=(\Xb(1),\dots,\Xb(n))$ for the concatenation of the first $n$ 
observations from all $N$ data streams. Let $\{g(\Xb(n)|\Xb^{n-1})\}_{n\ge 1}$ and $\{f_{\Bc}(\Xb_{n}|\Xb^{n-1})\}_{n\ge 1}$ be sequences of conditional densities of $\Xb(n)$ 
given $\Xb^{n-1}$, which may depend on $n$, i.e., $g=g_n$ and $f_\Bc=f_{\Bc,n}$.  For the general non-i.i.d.\  changepoint model, 
which we are interested in, the joint density $p(\Xb^n | H_{k,\Bc})$ under hypothesis $\Hyp_{k,\Bc}$ can be written as follows
\begin{align}
p(\Xb^n | H_{k,\Bc}) & = f_\infty(\Xb^n) = \prod_{t=1}^n g(\Xb(t)|\Xb^{t-1}) \quad \text{for}~~ \nu=k \ge n ,
\label{noniidmodelpre}
\\
p(\Xb^n | H_{k,\Bc}) & =  \prod_{t=1}^{k}  g(\Xb(t)|\Xb^{t-1}) \times \prod_{t=k+1}^{n}  f_{\Bc}(\Xb(t)|\Xb^{t-1})  \quad \text{for}~~ \nu=k < n,
\label{noniidmodelpost}
\end{align}
where $\Bc\subset \Pc$. Therefore, $g(\Xb_{n}|\Xb^{n-1})$ is the pre-change conditional density and $f_{\Bc}(\Xb_{n}|\Xb^{n-1})$ is the post-change conditional density given that the change 
occurs in the subset $\Bc$.

In most practical applications, the post-change distribution is not completely known -- it depends on an unknown (generally multi-dimensional) parameter $\theta\in\Theta$, 
so that the model \eqref{noniidmodelpost} 
may be treated only as a benchmark for a more practical case where the post-change densities $f_{\Bc}(\Xb(t)|\Xb^{t-1})$ are replaced by $f_{\Bc,\theta}(\Xb(t)|\Xb^{t-1})$, i.e.,
 \begin{align}
p(\Xb^n | H_{k,\Bc},\theta) & =  \prod_{t=1}^{k}  g(\Xb(t)|\Xb^{t-1}) \times \prod_{t=k+1}^{n}  f_{\Bc,\theta}(\Xb(t)|\Xb^{t-1})  \quad \text{for}~~ \nu=k < n.
\label{noniidmodelpostunknown}
\end{align}

In what follows we assume that the change point $\nu$ is a random variable independent of the observations with prior distribution 
$\pi_k=\Pb(\nu=k)$, $k=0,1,2,\dots$ with $\pi_k >0$ for $k\in\{0,1,2, \dots\}=\Zbb_+$. We will also assume that a change point may take negative values, which means that the change has 
occurred by the time the observations became available. However,  the detailed structure of the distribution $\Pb(\nu=k)$ for $k=-1,-2,\dots$ is not important. 
The only value which matters is the total probability $q=\Pb(\nu \le -1)$ of the change being in effect before the observations become available. 

\subsection{Double-mixture change detection procedures}\label{ssec:Proceduresunknown}

We begin by considering the most general scenario where the observations across streams are dependent and then go on tackling the scenario where the streams are mutually independent.

 \subsubsection{The general case}\label{sssec:General}
 
Let $\Lc_{\Bc,\theta}(n) = f_{\Bc,\theta}(\Xb(n)|\Xb^{n-1})/g(\Xb(n)|\Xb^{n-1})$. Note that in the general non-i.i.d.\ case the statistic 
$\Lc_{\Bc,\theta}(n)= \Lc_{\Bc,\theta}^{(k)}(n)$ depends on the change point $\nu=k$ since the post-change density  $ f_{\Bc,\theta}(\Xb(n)|\Xb^{n-1})$ may depend on $k$.
The likelihood ratio (LR) of the hypothesis ``$\Hyp_{k, \Bc}: \nu=k, \eta_i=1 ~ \text{for} ~ i \in \Bc$'' that the change occurs at $\nu=k$ in the subset of streams $\Bc$  against the no-change hypothesis 
``$\Hyp_\infty: \nu =\infty$'' based on the sample $\Xb^n=(\Xb(1),\dots,\Xb(n))$ is given by the product
\[
LR_{\Bc,\theta}(k, n) = \prod_{t=k+1}^{n}  \Lc_{\Bc,\theta}(t), \quad n > k
\]
and we set $LR_{\Bc,\theta}(k,n)=1$ for $n \le k$. For $\Bc \in \Pc$ and $\theta\in\Theta$, where $\Pc$ is an arbitrary class of subsets of $[N]$, define the statistic 
\begin{equation}\label{S_stat_noniid}
S_{\Bc,\theta}^\pi(n) = \frac{1}{\Pb(\nu \ge n)} \brcs{q LR_{\Bc,\theta}(0, n)+ \sum_{k=0}^{n-1} \pi_k \prod_{t=k+1}^{n} LR_{\Bc,\theta}(k, n)} , ~~ n \ge 1, ~~ S_{\Bc,\theta}^\pi(0)=q/(1-q),
\end{equation}
which is the Shiryaev-type statistic for detection of a change when it happens in a subset of streams $\Bc$ and the post-change parameter is $\theta$.

Next, let 
\[
\pb = \set{p_\Bc, \Bc \in \Pc}, \quad p_\Bc >0 ~ \forall ~ \Bc \in \Pc, \quad \sum_{\Bc \in \Pc} p_\Bc =1
\]
be the probability mass function on $[N]$ (mixing measure), and define the mixture statistic
\begin{equation}\label{Sp_stat_noniid}
S_{\pb,\theta}^\pi(n) = \sum_{\Bc \in \Pc} p_\Bc S_{\Bc,\theta}^\pi(n), \quad S_{\pb,\theta}^\pi(0)=q/(1-q).
\end{equation}
This statistic can be also represented as
\begin{equation}\label{Sp_stat_noniid1}
\begin{split}
S_{\pb,\theta}^\pi(n)& = \frac{1}{\Pb(\nu \ge n)} \brcs{q  \Lambda_{\pb,\theta}(0,n) + \sum_{k=0}^{n-1} \pi_k \Lambda_{\pb,\theta}(k,n)}  ,
\end{split}
\end{equation}
where
\[
\Lambda_{\pb,\theta}(k,n) =  \sum_{\Bc \in \Pc} p_\Bc  LR_{\Bc,\theta}(k, n)
\]
is the mixture LR.

When the parameter $\theta$ is unknown there are two conventional approaches -- either to maximize or average (mix) over  $\theta$.  Introduce a mixing measure 
$W(\theta)$, $\int_{\Theta} \rm{d} W(\theta) =1$, which can be interpreted as a prior distribution and
define the double LR-mixture (average LR) 
\begin{equation}\label{ALR} 
\begin{split}
\Lambda_{\pb,W}(k,n) & =  \int_\Theta \sum_{\Bc \in \Pc} p_\Bc  LR_{\Bc,\theta}(k, n)  \, \mrm{d} W(\theta)
\\
& = \int_\Theta \Lambda_{\pb,\theta}(k,n)  \, \mrm{d} W(\theta), \quad k < n
\end{split}
\end{equation}
and the double-mixture Shiryaev-type statistic 
\begin{equation}\label{DMS_stat}
\begin{split}
S_{\pb,W}^{\pi}(n) & = \int_\Theta \sum_{\Bc \in \Pc} p_\Bc S_{\Bc,\theta}^\pi(n) \, \mrm{d} W(\theta)
\\
& = \frac{1}{\Pb(\nu \ge n)} \brcs{q  \Lambda_{\pb,W}(0,n) + \sum_{k=0}^{n-1} \pi_k \Lambda_{\pb,W}(k,n)}  .
\end{split}
\end{equation}
The corresponding double-mixture LR-based detection procedure is given by the stopping rule which is the first time $n\ge 1$ such that the  statistic $S_{\pb,W}^{\pi}(n)$ hits the level $A>0$, i.e.,
\begin{equation}\label{DMS_def}
T_A^W=\inf\set{n \ge 1: S_{\pb,W}^\pi(n) \ge A}.
\end{equation}

Another popular statistic for detecting a change from  $\{g(\Xb(n)|\Xb^{n-1})\}$ to $\{f_{\Bc,\theta}(\Xb(n)|\Xb^{n-1})\}$, which has certain optimality properties 
\cite{PollakTartakovsky-SS09,tartakovsky-tpa11,PolunTartakovskyAS09,TNB_book2014}, is the generalized Shiryaev--Roberts (SR) statistic
\begin{equation}\label{SR_stat_noniid}
R_{\Bc,\theta}(n) =\omega LR_{\Bc,\theta}(0, n) + \sum_{k=0}^{n-1} LR_{\Bc,\theta}(k, n) , \quad n \ge 1, ~~ R_{\Bc,\theta}(0) =\omega
\end{equation}
with a non-negative head-start $\omega \ge 0$. For a fixed value of $\theta$, introduce the mixture statistic 
\begin{equation}\label{MSR_stat}
\begin{split}
R_{\pb,\theta}(n)&= \sum_{\Bc \in \Pc} p_\Bc R_{\Bc,\theta}(n)
\\
& = \omega  \Lambda_{\pb,\theta}(0,n) +  \sum_{k=0}^{n-1}  \Lambda_{\pb,\theta}(k,n)  \quad n \ge 1, ~~ R_{\pb,\theta}(0)=\omega ,
\end{split}
\end{equation}
and the generalized double-mixture SR statistic
\begin{equation}\label{DMSR_stat}
\begin{split}
R_{\pb,W}(n)&= \int_\Theta \sum_{\Bc \in \Pc} p_\Bc R_{\Bc,\theta}(n) \, \mrm{d} W(\theta)
\\
& = \omega  \Lambda_{\pb,W}(0,n) +  \sum_{k=0}^{n-1}  \Lambda_{\pb,W}(k,n),  \quad n \ge 1, ~~ R_{\pb,W}(0)=\omega 
\end{split}
\end{equation}
(with a non-negative head-start $\omega$) as well as the corresponding stopping rule  
\begin{equation}\label{DMSR_def}
\wtT_A^W=\inf\set{n \ge 1: R_{\pb,W}(n) \ge A}, \quad A>0.
\end{equation}

Note that we consider a very general stochastic model where not only the observations in streams may be dependent and non-identically distributed, but also the streams may be mutually dependent. 
In this very general case, computing statistics $S_{\pb,W}^\pi(n)$ and $R_{\pb,W}(n)$ is problematic even when the statistics in data streams $S_i^\pi(n)$ and $R_i(n)$, $i=1,\dots,N$, 
can be computed. The computational problem becomes manageable when the data between data streams are independent, as discussed in the next subsection. 

\subsubsection{Independent streams}\label{sssec:indepchannels}

Consider now a special scenario where the data across streams are independent. Note that in the case of independent streams the post-change parameters can be assumed different in streams, i.e., 
$\theta=\theta_i$ for the $i$th stream, $i=1,\dots,N$. In contrast to the general case of dependent streams, this does not lead to an additional complication. Thus, we have
\begin{equation}\label{ind}
\begin{split}
p(\Xb^n | H_{k,\Bc}, \teb_\Bc) & = f_\infty(\Xb^n) = \prod_{t=1}^n \prod_{i=1}^N g_i(X_i(t)|\Xb_i^{t-1}) \quad \text{for}~ \nu=k \ge n ,
\\
p(\Xb^n | H_{k,\Bc},\teb_\Bc) & =  \prod_{t=1}^{k}  \prod_{i=1}^N g_i(X_i(t)|\Xb_i^{t-1}) \times 
\\
& \quad \prod_{t=k+1}^{n}  \prod_{i\in \Bc} f_{i,\theta_i}(X_i(t)|\Xb_i^{t-1}) \prod_{i\notin \Bc} g_i(X_i(t)|\Xb_i^{t-1}) 
\quad \text{for}~ \nu=k < n,
\end{split}
\end{equation}
where $g_i(X_i(t)|\Xb_i^{t-1})$ and $f_{i,\theta_i}(X_i(t)|\Xb_i^{t-1})$ are conditional pre- and post-change densities in the $i$th data stream, respectively, and $\teb_\Bc=(\theta_i,i\in\Bc)$.
So the LRs are
\begin{equation}\label{LRind}
LR_{\Bc,\theta_\Bc}(k, n) = \prod_{i\in \Bc} LR_{i,\theta_i}(k,n), \quad LR_{i,\theta_i}(k,n)= \prod_{t=k+1}^{n}  \Lc_{i,\theta_i}(t), \quad n > k,
\end{equation}
where $\Lc_{i,\theta_i}(t)=f_{i,\theta_i}(X_i(t)|\Xb_i^t)/g_i(X_i(t)|\Xb_i^t)$.

Assume in addition that the mixing measure is such that
\[
p_\Bc= C(\Pc_K) \prod_{i\in \Bc} p_i, \quad C(\Pc_K) = \brc{\sum_{\Bc \in \Pc_K} \prod_{i \in \Bc} p_i}^{-1}. 
\]
Then the mixture LR is
\[
\Lambda_{\pb,\teb}(k,n) = C(\Pc_K) \sum_{i=1}^K \sum_{\Bc\in \widetilde{\Pc}_i} \prod_{j\in\Bc} p_j LR_{j,\theta_i}(k,n),
\]
and its computational complexity is polynomial in the number of data streams. Moreover, in the special, perhaps most interesting and difficult case of $K=N$ and $p_j=p$, we obtain
\begin{equation}\label{mixLRind}
\Lambda_{\pb,\teb}(k,n) = C(\Pc_N) \brcs{\prod_{i=1}^N \brc{1+ p LR_{i,\theta_i}(k,n)} -1},
\end{equation}
so its computational complexity is only $O(N)$. The representation \eqref{mixLRind} corresponds to the case when each stream is affected independently with probability $p/(1+p)$, 
the assumption that was made in \cite{Xie&Siegmund-AS13}.

\section{Asymptotic optimality problems and assumptions}\label{sec:Problem}

Let $\EV_{k,\Bc,\theta}$ and $\EV_\infty$ denote expectations under $\Pb_{k,\Bc,\theta}$ and $\Pb_\infty$, respectively, where $\Pb_{k,\Bc,\theta}$ corresponds to 
model \eqref{noniidmodelpostunknown} with an unknown parameter $\theta\in\Theta$. Define the probability measure 
$\Pb^\pi_{\Bc,\theta} (\Ac\times \mc{K})=\sum_{k\in \mc{K}}\,\pi_{k} \Pb_{k,\Bc,\theta}\left(\Ac\right)$  under which the change point $\nu$ has distribution $\pi=\{\pi_k\}$ and the model for the observations
is of the form \eqref{noniidmodelpre},\eqref{noniidmodelpostunknown}, i.e., $\Xb(t)$ has conditional density $g(\Xb(t)| \Xb^{t-1})$ if $\nu \le k$ and conditional density 
$f_{\Bc,\theta}(\Xb(t)| \Xb^{t-1})$ if $\nu > k$ and the change occurs in the subset $\Bc$ with the parameter $\theta$. Let $\Eb^\pi_{\Bc,\theta}$ denote the corresponding expectation.

For $r \ge 1$, $\nu=k \in \Zbb_+$, $\Bc\in\Pc$, and $\theta\in\Theta$ introduce the risk associated with the conditional $r$th moment of the detection delay
$\Rc^r_{k,\Bc,\theta}(T)=  \EV_{k, \Bc,\theta}\left[(T-k)^r\,|\, T> k \right]$.
In a Bayesian setting, the risk associated with the moments of delay to detection  is 
\begin{equation} \label{Riskdef}
\Rca^r_{\Bc,\theta}(T): = \Eb^\pi_{\Bc,\theta} [ (T-\nu)^r | T> \nu]= 
\frac{{\displaystyle\sum_{k=0}^\infty} \pi_k \Rc^r_{k,\Bc,\theta}(T)\Pb_\infty( T >k)}{1-\PFA( T)} ,
\end{equation}
where 
\begin{equation} \label{PFAdef}
\pfa(T)=\Pb^\pi_{\Bc,\theta}( T \le \nu)= \sum_{k=0}^\infty \pi_k \Pb_\infty( T \le k) 
\end{equation}
is the weighted probability of false alarm (PFA) that corresponds to the risk associated with a false alarm. Note that in \eqref{Riskdef} and \eqref{PFAdef} we used the fact that
$\Pb_{k,\Bc,\theta}(T \le k) = \Pb_\infty(T\le k)$ since the event $\{T \le k\}$ depends on the observations $\Xb_1,\dots \Xb_k$ generated by the pre-change probability measure $\Pb_\infty$ since
by our convention $\Xb_k$ is the last pre-change observation if $\nu=k$.

In this article, we are interested  in the Bayesian (constrained) optimization problem
\begin{equation}\label{sec:PrbfBayes}
\inf_{\{T: \PFA(T) \le \alpha\}}\,\Rca_{\Bc,\theta}^r(T)  \quad \text{for all} ~ \Bc\in \Pc, ~~ \theta\in\Theta.
\end{equation}
However, in general this problem is intractable for every value of the PFA $\alpha\in (0, 1)$. So we will focus on the asymptotic problem assuming that the PFA $\alpha$ approaches zero. 
Specifically, we will be interested in proving that the double-mixture detection procedure $T_A^W$ is first-order uniformly asymptotically optimal for all possible subsets $\Bc\in \Pc$ where the change 
may occur and all parameter values $\theta\in\Theta$ , i.e.,
\begin{equation}\label{FOAOdef}
\lim_{\alpha\to0} \frac{\displaystyle\inf_{T\in\class}\Rca_{\Bc,\theta}^r(T)}{\Rca_{\Bc,\theta}^r(T_A^W)} =1   \quad \text{for all} ~ \Bc\in \Pc, ~~ \theta\in\Theta
\end{equation} 
and
\begin{equation}\label{FOAOunifdef}
\lim_{\alpha\to0} \frac{\displaystyle\inf_{T\in\class}\Rc_{k, \Bc,\theta}^r(T)}{\Rc_{k, \Bc,\theta}^r(T_A^W)} =1   \quad \text{for all} ~ \Bc\in \Pc,~ \theta\in\Theta ~ \text{and all}~ \nu=k\in \Zbb_+,
\end{equation} 
where $\class=\{T: \PFA(T) \le \alpha\}$ is the class of detection procedures for which the PFA does not exceed a prescribed number $\alpha \in (0,1)$ and $A=A_\alpha$ is suitably selected.

First-order asymptotic optimality properties of the double-mixture SR-type detection procedure $\wtT_A^W$ under certain conditions will be also established.

Instead of the constrained optimization problem \eqref{sec:PrbfBayes} one may be also interested in the unconstrained Bayes problem with the average (integrated) risk function
\begin{equation} \label{Averrisk}
\begin{split}
\rho_{\pi, \pb, W}^{c,r}(T) &= \Eb\brcs{\Ind{T\le \nu} + c \, (T-\nu)^r \Ind{T > \nu}} 
\\
&=  \PFA(T) + c \, \sum_{\Bc\in\Pc} p_{\Bc} \int_\Theta \Eb_{\Bc,\theta}^\pi[(T-\nu)^+]^r \, \drm W(\theta),
\end{split}
\end{equation}
where $c>0$ is the cost of delay per unit of time  and  $r \ge 1$. An unknown post-change parameter $\theta$ and an unknown location of the change pattern $\Bc$ are now assumed random and 
the weight functions $W(\theta)$ and $p_\Bc$ are interpreted as the prior distributions of $\theta$ and $\Bc$, respectively.  The first-order asymptotic problem is
\begin{equation}\label{FOAOdefAvRisk}
\lim_{c \to0} \frac{\displaystyle\inf_{T \ge 0}\rho_{\pi, \pb, W}^{c,r}(T)}{\rho_{\pi, \pb, W}^{c,r}(T_A^W)} =1,   
\end{equation} 
where threshold $A=A_{c,r}$ that depends on the cost $c$ should be suitably selected.

While we consider a general prior and a very general stochastic model for the observations in streams and between streams, to study asymptotic optimality properties we still 
need to impose certain constraints on the prior distribution $\{\pi_k\}$ and on the general stochastic model \eqref{noniidmodelpre}--\eqref{noniidmodelpost} that guarantee 
asymptotic stability of the detection statistics as the sample size increases. 

In what follows, we assume that the prior distribution $\pi^\alpha=\{\pi_k^\alpha\}$ may depend on $\alpha$ and the following condition is imposed:
\vspace{2mm}

\noindent $\mb{CP} \mb{1}$. {\em  For some $0 \le \mu_\alpha <\infty$ and  $0 \le \mu <\infty$},
\begin{equation}\label{Prior}
\lim_{n\to\infty}\frac{1}{n}\abs{\log \sum_{k=n+1}^\infty \pi_k^\alpha} = \mu_\alpha \quad \text{and} \quad \lim_{\alpha\to0} \mu_\alpha = \mu. 
\end{equation}
The class of prior distributions satisfying condition $\mb{CP} \mb{1}$  will be denoted by $\Cb(\mu)$.

For establishing asymptotic optimality properties of change detection procedures we will assume in addition that the following two condition hold:
\vspace{2mm}

\noindent $\mb{CP} {\mb 2}$. {\em If $\mu_\alpha>0$ for all $\alpha$ and $\mu=0$, then $\mu_\alpha$ approaches zero at such rate that for some $r\ge 1$} 
\begin{equation}\label{Prior1}
\lim_{\alpha\to 0} \frac{{\sum_{k=0}^\infty  \pi_k^\alpha |\log \pi_k^\alpha|^r}}{|\log \alpha|^r} = 0.
\end{equation} 
\noindent $\mb{CP} {\mb 3}$. {\em For all $k \in \Zbb_+$} 
\begin{equation}\label{Prior2}
\lim_{\alpha\to 0} \frac{|\log \pi_k^\alpha|}{|\log \alpha|} = 0.
\end{equation}

Note that if $\mu >0$, then the prior distribution has an exponential right tail (asymptotically) with the positive parameter $\mu$, in which case, condition \eqref{Prior1} holds since 
$\lim_{\alpha\to 0} \sum_{k=0}^\infty  \pi_k^\alpha |\log \pi_k^\alpha|^r<\infty$ for all $r>0$. If $\mu=0$, the distribution has a heavy tail (at least asymptotically) and
we cannot allow this distribution to have a too heavy tail, which will generate very large time intervals between change points.
This is guaranteed by condition  $\mb{CP} {\mb 2}$. Note that condition  $\mb{CP} {\mb 1}$ excludes light-tail distributions with unbounded hazard rates for which $\mu=\infty$ 
and the time-intervals with a change point are very short (e.g., Gaussian-type or Weibull-type with the shape parameter $\kappa>1$). 
In this case, prior information dominates information obtained from the observations, the change can be easily detected at early stages, and the asymptotic analysis is impractical.  Note also that 
if the prior distribution does not depend on $\alpha$, then in condition $\mb{CP} \mb{1}$ $\mu_\alpha=\mu$ and $\mb{CP} \mb{2}$ holds when $\sum_{k=0}^\infty \pi_k |\log\pi_k|^r < \infty$ 
for some $r\ge 1$. These conditions were used in \cite{TartakovskyIEEEIT2017}.
 
For $\Bc \in \Pc$ and $\theta\in\Theta$, introduce the log-likelihood ratio (LLR) process $\{\lambda_{\Bc,\theta}(k, n)\}_{n \ge k+1}$ between the hypotheses 
``$\Hyp_{k,\Bc}, \theta$'' ($k=0,1, \dots$) and $\Hyp_\infty$:
$$
\lambda_{\Bc,\theta}(k, n) = \sum_{t=k+1}^{n}\, \log \frac{f_{\Bc,\theta}(\Xb(t) | \Xb^{t-1})}{g(\Xb(t)|\Xb^{t-1})}, \quad n >k 
$$
($\lambda_{\Bc,\theta}(k, n)=0$ for $n \le k$). 

Define
\[
\beta_{M,k}(\varepsilon,\Bc,\theta)=\Pb_{k,\Bc,\theta}\set{\frac{1}{M}\max_{1 \le n \le M} \lambda_{\Bc,\theta}(k, k+n) \ge (1+\varepsilon) I_{\Bc,\theta}}
\]
and for $\delta>0$ define $\Gamma_{\delta,\theta}=\{\vartheta\in\Theta\,:\,\vert \vartheta-\theta\vert<\delta\}$ and
\[
\Upsilon_{r}(\varepsilon, \Bc, \theta) = 
\sum_{n=1}^\infty n^{r-1} \sup_{k\in \Zbb_+} \Pb_{k,\Bc, \theta}\set{\frac{1}{n} \inf_{\vartheta\in\Gamma_{\delta,\theta}}\lambda_{\Bc,\vartheta}(k, k+n) < I_{\Bc,\theta}  - \varepsilon}.
\]
Regarding the general model for the observations \eqref{noniidmodelpre}, \eqref{noniidmodelpostunknown} we assume that the following two conditions are satisfied:
\vspace{2mm}

\noindent $\C_{1}$. {\em  There exist positive and finite numbers $I_{\Bc,\theta}$ ($\Bc \in \Pc$, $\theta\in \Theta$) such that the LLR $n^{-1}\lambda_{\Bc,\theta}(k,k+n) \to  I_{\Bc,\theta}$ in
$\Pb_{k, \Bc,\theta}$-probability and for any $\varepsilon >0$}
\begin{equation}\label{sec:Pmax}
\lim_{M\to\infty}  \beta_{M,k}(\varepsilon,\Bc,\theta) =0 \quad \text{for all}~ k\in \Zbb_+, \Bc\in \Pc, \theta\in\Theta ;
\end{equation}

\noindent $\C_{2}$. {\em  For any $\varepsilon>0$ there exists $\delta=\delta_{\varepsilon}>0$ such that $W(\Gamma_{\delta,\theta})>0$
and for any $\varepsilon>0$ and some $r\ge 1$}
\begin{equation}\label{rcompLeft}
\Upsilon_{r}(\varepsilon, \Bc, \theta) < \infty \quad \text{for all}~  \Bc\in\Pc, \theta\in\Theta .
\end{equation}

\section{Asymptotic lower bounds for moments of the detection delay and average risk function}\label{sec:LowerBounds}

In order to establish asymptotic optimality of detection procedures we first obtain, under condition $\C_1$, asymptotic (as $\alpha\to0$) lower bounds for moments 
of the detection delay  $\Rca^r_{\Bc,\theta}(T)=\Eb^\pi_{\Bc,\theta}\brcs{\brc{ T-\nu}^r| T >\nu}$ and $\Rc^r_{k, \Bc,\theta}= \Eb_{k, \Bc,\theta}\brcs{\brc{ T-k}^r| T >k}$
of any detection procedure $T$ from class $\class$. 
In the following sections, we  show that under 
condition $\C_2$ these bounds are attained for the double-mixture procedure $T_{A}^W$  uniformly for all $\Bc\in\Pc$ and $\theta\in\Theta$ and
that the same is true for the double-mixture procedures $\wtT_A^W$ when the prior distribution is either heavy-tailed or
has an exponential tail with a small parameter $\mu$. We also establish the asymptotic lower bound for the integrated risk $\rho_{\pi, \pb, W}^{c,r}(T)$ as $c\to0$ in the class of all Markov times and show that it is
attained by the double-mixture procedures $T_{A}^W$ and $\wtT_A^W$.

Define
\begin{equation}\label{R}
\Rca_{\pb, W}^{r}(T) = \sum_{\Bc\in\Pc} p_\Bc \int_\Theta  \Rca_{\Bc,\theta}^r(T) \, \drm W(\theta) 
\end{equation}
and
\begin{equation}\label{Dr}
D_{\mu,r} =  \sum_{\Bc\in\Pc} p_\Bc \int_\Theta \brc{\frac{1}{I_{\Bc,\theta} + \mu}}^r \, \drm W(\theta).
\end{equation}

Asymptotic lower bounds for all positive moments of the detection delay and the integrated risk are specified in the following theorem.

\begin{theorem}\label{Th:LB}
Let, for some $\mu\ge 0$, the prior distribution belong to class $\Cb(\mu)$. 
Assume that for some positive and finite numbers $I_{\Bc,\theta}$ ($\Bc\in\Pc$, $\theta\in \Theta$) condition $\C_1$ holds. Then for all $r >0$ and all 
$\Bc\in\Pc$, $\theta\in\Theta$
\begin{equation}\label{LBkinclass}
\liminf_{\alpha\to0} \frac{{\displaystyle\inf_{ T\in\class}}  \Rc^r_{k,\Bc,\theta}(T)}{|\log \alpha|^r} \ge  \frac{1}{(I_{\Bc,\theta}+\mu)^r}, \quad  k\in \Zbb_+,
\end{equation}
\begin{equation}\label{LBinclass}
\liminf_{\alpha\to0} \frac{{\displaystyle\inf_{ T\in\class}} \Rca^r_{\Bc,\theta}(T)}{|\log \alpha|^r} \ge  \frac{1}{(I_{\Bc,\theta}+\mu)^r} ,
\end{equation}
and for all $r>0$
\begin{equation}\label{LBAR}
\liminf_{c\to0} \frac{{\displaystyle\inf_{T\ge 0}}\rho_{\pi, \pb, W}^{c,r}(T)}{c |\log c|^r} \ge D_{\mu,r}.
\end{equation}
\end{theorem}

\begin{proof}
The methodology of the proof is essentially analogous to that used in the proofs of the lower bounds in Tartakovsky~\cite{TartakovskyIEEEIT2017,TartakovskyIEEEIT2018} for a single stream change
detection problem with slightly different assumptions on the
prior distribution. In particular, since the vector $(\Bc,\theta)=\tilde\theta$ is also an unknown parameter, the lower bound \eqref{LBkinclass} follows from Lemma~1 in ~\cite{TartakovskyIEEEIT2018} 
if $\mu_\alpha=\mu$ for all $\alpha$ (i.e., when $\pi_k^\alpha$ does not depend on $\alpha$) and from Lemma~3 in \cite{TartakovskyIEEEIT2018}
if $\mu_\alpha \to \mu=0$ by simple replacing $\theta$ by $\tilde\theta$.  A generalization of the proof under condition $\mb{CP} \mb{1}$ introduced in the present article has several technical 
details that are presented
below. We omit certain intermediate inequalities, which follow from the proofs given in~\cite{TartakovskyIEEEIT2017,TartakovskyIEEEIT2018}.

 For $\varepsilon\in(0,1)$ and $\delta>0$, define 
\[
N_{\alpha}=N_{\alpha}(\varepsilon,\delta, \Bc,\theta)=  \frac{(1-\varepsilon)|\log\alpha|}{I_{\Bc,\theta}+\mu+\delta}
\] 
and let $\Pi^\alpha_k = \Pb(\nu > k)$ ($\Pi_{-1}^\alpha= 1-q^\alpha$).
Using the fact  $\Pb_{k, \Bc,\theta}( T>k)=\Pb_\infty (T>k)$ and Chebyshev's inequality, as in (A.1) in \cite{TartakovskyIEEEIT2018}, we obtain
\begin{equation} \label{A1}
\begin{split}
&\inf_{ T\in\class}\Rc_{k,\Bc,\theta}^r(T) 
\\
&\ge N_{\alpha}^r\brcs{\inf_{T\in\class}\Pb_{\infty}( T>k) - \sup_{ T\in \class}\Pb_{k, \Bc,\theta}(k <  T \le k+N_{\alpha})}
\\
&\ge  N_{\alpha}^r\brcs{1- \alpha/ \Pi_{k-1}^\alpha - \sup_{ T\in \class}\Pb_{k, \Bc,\theta}(k <  T \le k+N_{\alpha})},
\end{split}
\end{equation}
where we used the inequality
\begin{equation}\label{Psup}
\inf_{T\in\class} \Pb_\infty(T > k) \ge 1- \alpha/\Pi_{k-1}^\alpha, \quad k \in \Zbb_+,
\end{equation}
which follows from the fact that for any stopping rule $T\in\class$,
\[
\alpha \ge \sum_{i=k}^\infty \pi_i^\alpha \Pb_\infty(T \le i) \ge \Pb_\infty(T \le k) \sum_{i=k}^\infty \pi_i^\alpha .
\]

Thus, to prove the lower bound \eqref{LBkinclass} it suffices to show that for arbitrary small $\varepsilon$ and $\delta$ and every fixed $k \in \Zbb_+$
\begin{equation}\label{Pksupclasszero}
 \sup_{ T\in \class}\Pb_{k, \Bc,\theta}(k <  T \le k+N_{\alpha})  \to 0 \quad \text{as}~ \alpha\to0.
\end{equation}
Introduce
\[
\begin{aligned} 
U_{N_\alpha,k}(T) &= e^{(1+\varepsilon) I_{\Bc,\theta} N_{\alpha}}\Pb_\infty\brc{k <  T \le k+ N_{\alpha}}, 
\end{aligned}
\]
By inequality (3.6) in \cite{TartakovskyVeerTVP05}, for any stopping time $T$,
\begin{equation}\label{PkTupper}
 \Pb_{k, \Bc,\theta}\brc{k <  T \le  k+ N_{\alpha}} \le  U_{N_\alpha,k}( T)  + \beta_{N_\alpha,k}(\varepsilon,\Bc,\theta).
\end{equation}
Using inequality  \eqref{Psup} and the fact that by condition \eqref{Prior}, for all sufficiently large $N_{\alpha}$ (small $\alpha$), there exists a (small) $\delta$ such that
\[
\frac{|\log \Pi^\alpha_{k-1+N_\alpha}|}{ k-1+N_{\alpha}} \le \mu + \delta,
\]
we obtain that for all sufficiently small $\alpha$
\begin{align*} 
U_{N_\alpha,k}( T)  & \le e^{(1+\varepsilon) I_{\Bc,\theta} N_{\alpha}} \Pb_\infty( T \le k+N_{\alpha}) 
\\
&\le \alpha \, e^{(1+\varepsilon) I_{\Bc,\theta}N_{\alpha}}/\Pi^\alpha_{k-1+N_\alpha} 
\\
& \le \exp\set{(1+\varepsilon) I_{\Bc,\theta} N_{\alpha}-|\log\alpha| +( k-1+N_{\alpha}) \frac{|\log\Pi^\alpha_{k-1+N_\alpha}|}{ k-1+N_{\alpha}}} 
\\
& \le  \exp\set{(1+\varepsilon) I_{\Bc,\theta} N_{\alpha}-|\log\alpha| +( k-1+N_{\alpha}) (\mu+\delta)} 
\\
&= \exp\set{-\frac{I_{\Bc,\theta}\varepsilon^2 + (\mu+\delta) \varepsilon}{I_{\Bc,\theta}+\mu + \delta}|\log\alpha| + (\mu+\delta) (k-1)},
\end{align*}
where the last term is less or equal to 
\[
\exp\set{- \varepsilon^2 |\log\alpha| + (\mu+\delta) (k-1)} := \overline{U}_{\alpha,k}(\varepsilon,\delta)
\]
for all $\varepsilon \in (0,1)$. Thus,  for all $\varepsilon \in (0,1)$,
\begin{equation} \label{IneqU}
U_{N_\alpha,k}( T)   \le  \overline{U}_{\alpha,k}(\varepsilon,\delta) .
\end{equation}
Since $\overline{U}_{\alpha,k}(\varepsilon,\delta)$ does not depend on the stopping time $T$ and the value of $\overline{U}_{\alpha,k}(\varepsilon,\delta)$  goes to $0$ as $\alpha\to0$ 
for any fixed $k\in \Zbb_+$ and any $\varepsilon>0$ and $\delta>0$, it follows that 
\begin{equation}\label{IneqsupU}
\sup_{T\in\class} U_{N_\alpha,k}( T) \to 0 \quad \text{as}~ \alpha\to0 ~~ \text{for any fixed}~ k \in\Zbb_+.
\end{equation}
Also, by condition $\C_1$, $\beta_{N_\alpha,k}(\varepsilon,\Bc,\theta)\to0$ for all $\varepsilon>0$, $k\in\Zbb_+$, $\Bc\in\Pc$, $\theta\in\Theta$, and therefore, \eqref{Pksupclasszero} holds. This completes
 the proof of the lower bound \eqref{LBkinclass}.   
 
We now prove the lower bound \eqref{LBinclass}. Let $K_\alpha=K_{\alpha}(\varepsilon) = 1+\lfloor \varepsilon^2 \, |\log \alpha| \rfloor$.
Using \eqref{PkTupper} and \eqref{IneqU}, we obtain 
\begin{align}\label{ProbUpper1}
& \sup_{ T\in \class}\Pb^\pi_{\Bc,\theta}(\nu  <  T  \le \nu +  N_{\alpha}) 
= \sum_{k=0}^\infty  \pi_k^\alpha  \sup_{ T\in\class}  \Pb_{k, \Bc,\theta}\brc{k <  T \le k+ N_{\alpha}} \nonumber
 \\
 &\le \Pb(\nu > K_{\alpha}) +  \sum_{k=0}^{K_{\alpha}}  \pi_k^\alpha \beta_{N_\alpha,k}(\varepsilon,\Bc,\theta)  + 
 \max_{0 \le k \le K_{\alpha}} \sup_{ T\in\class} U_{N_\alpha,k}( T) \nonumber
 \\
 & \le \Pb(\nu > K_{\alpha}) +  \sum_{k=0}^{K_{\alpha}}  \pi_k^\alpha \beta_{N_\alpha,k}(\varepsilon,\Bc,\theta) + \overline{U}_{\alpha,K_\alpha}(\varepsilon,\delta) .
 \end{align}
 If condition \eqref{Prior} holds with $\mu>0$, then $\log \Pb(\nu > K_{\alpha}) \sim - \mu \, \varepsilon^2 |\log \alpha|\to -\infty$ as $\alpha \to0$, so the probability 
 $\Pb(\nu > K_{\alpha})$ goes to $0$ as $\alpha \to 0$ and the same is true if $\mu_\alpha>0$ for all $\alpha$ and $\mu=0$ since, in this case, 
 $\log \Pb(\nu > K_{\alpha}) \sim -\varepsilon^3 |\log \alpha| \to -\infty$ for a sufficiently small $\varepsilon$. If $\mu_\alpha=\mu=0$ for all $\alpha$,  then $\Pb(\nu > K_{\alpha})\to 0$ as well since
\[
\sum_{k= K_{\alpha}}^\infty  \pi_k^0  \xra[\alpha\to0]{} 0. 
\]
By condition $\C_1$, the second term in \eqref{ProbUpper1} also goes to zero.  Obviously, $\overline{U}_{\alpha,K_\alpha}(\varepsilon,\delta)\to 0$ as $\alpha\to0$, and therefore, 
all three terms go to zero as $\alpha\to0$ for all $\varepsilon,\delta>0$, $\Bc\in\Pc$, $\theta\in\Theta$, so that
\begin{equation}\label{Psupclasszero}
 \sup_{ T\in \class}\Pb_{\Bc,\theta}^\pi(\nu <  T \le \nu+ N_{\alpha})  \to 0 \quad \text{as}~ \alpha \to 0.
\end{equation}

Using Chebyshev's inequality, similarly to \eqref{A1} we obtain that 
\begin{equation} \label{IneqRcainf}
\begin{split}
\inf_{ T\in\class}\Rca_{\Bc,\theta}^r(T) 
\ge  N_{\alpha}^r\brcs{1- \alpha - \sup_{ T\in \class}\Pb_{\Bc,\theta}^\pi (\nu <  T \le  \nu+N_{\alpha})}.
\end{split}
\end{equation}
By \eqref{Psupclasszero} and \eqref{IneqRcainf}, asymptotically as $\alpha\to0$
\[
\inf_{ T\in\class}\Rca_{\Bc,\theta}^r(T) \ge  \brcs{\frac{(1-\varepsilon)|\log\alpha|}{I_{\Bc,\theta}+\mu+\delta}}^r (1+o(1)),
\]
where $\varepsilon$ and $\delta$ can be arbitrarily small, so that the lower bound \eqref{LBinclass}  follows.  

In order to prove the lower bound \eqref{LBAR} let us define the function
\begin{equation}\label{Gfunction}
G_{c,r}(A) = \frac{1}{A} + D_{\mu,r} c (\log A)^r.
\end{equation}
It is easily seen that $\min_{A>0}G_{c,r}(A)=G_{c,r}(A_{c,r})$, where $A_{c,r}$ satisfies the equation
\begin{equation*} 
r D_{\mu,r} A (\log A)^{r-1}-1/c =0,
\end{equation*}
and goes to infinity as $c\to0$ so that $\log A_{c,r}\sim |\log c|$ and that
\[
\lim_{c\to 0} \frac{G_{c,r}(A_{c,r})}{D_{\mu,r} \, c \, |\log c|^r} =1.
\]
Thus, it suffices to prove that
\begin{equation}\label{needtoprove}
\frac{\displaystyle\inf_{T\ge 0} \rho_{\pb, W}^{c,r}(T)}{G_{c,r}(A_{c,r})} \ge 1 +o(1) \quad \text{as}~ c \to 0.
\end{equation}
If \eqref{needtoprove} is wrong, then there is a stopping rule $T_{c}$ such that
\begin{equation} \label{assumptionTc}
\frac{\rho_{\pb, W}^{c,r}(T_c)}{G_{c,r}(A_{c,r})} < 1 +o(1) \quad \text{as}~ c \to 0 .
\end{equation}
Let $\alpha_c=\PFA(T_c)$. Since 
\[
\alpha_c\le \rho_{\pb, W}^{c,r}(T_c) < G_{c,r}(A_{c,r})(1+o(1))\to 0 \quad \text{as}~~c\to0,
\]
it follows that $\alpha_c \to 0$ as $c\to 0$. Using inequality \eqref{LBinclass}, we obtain that as $\alpha_c \to 0$
\[
\Rca_{\pb,W}^r(T_c)  \ge D_{\mu,r}  |\log \alpha_c|^r (1+o(1)),
\]
and hence, as $c \to0$,
\begin{align*}
\rho_{\pb, W}^{c,r}(T_c)& = \alpha_c + c\, (1-\alpha_c) \Rca_{\pb,W}^r(T_c) 
\\
& \ge \alpha_c + c \, D_{\mu,r}  |\log \alpha_c|^r (1+o(1)) .
\end{align*}
Thus,
\[
\frac{\rho_{\pb, W}^{c,r}(T_c)}{G_{c,r}(A_{c,r})} \ge \frac{G_{c,r}(1/\alpha_c) +c \, D_{\mu,r} |\log \alpha_c|^r o(1)}{\min_{A>0} G_{c,r}(A)} \ge 1+o(1),
\]
which contradicts~\eqref{assumptionTc}. Hence, \eqref{needtoprove} follows and the proof of \eqref{LBAR} is complete.
\end{proof}

\section{First-order asymptotic optimality of the detection procedures $T_A^W$ and $\wtT_A^W$ in class $\class$}\label{sec:AODMS} 

We now proceed with establishing asymptotic optimality properties of the double-mixture detection procedures $T_A^W$ and $\wtT_A^W$ in class $\class$ as $\alpha\to0$.

\subsection{First-order asymptotic optimality of the procedure $T_A^W$}\label{ssec:AOTWA}

The following lemma provides the upper bound for the PFA of the double-mixture procedure $T_A^W$ defined in \eqref{DMS_def}.

\begin{lemma}\label{Lem:PFADMS} 
For all  $A>q/(1-q)$ and any prior distribution of the change point, the PFA of the procedure  $T_{A}^W$ satisfies the inequality 
\begin{equation}\label{PFADMSineq}
\PFA(T_A^W) \le 1/(1+A),
\end{equation} 
so that if $A=A_\alpha=(1-\alpha)/\alpha$, then $\PFA(T_{A_\alpha}^W) \le \alpha$ for $0<\alpha < 1-q$, i.e., $T_{A_\alpha}^W \in \class$. 
\end{lemma}

\proof
Using the Bayes rule and the fact that  $\Lambda_{\pb,W}(k,n)=1$ for $k \ge n$, we obtain
\begin{align*}
& \Pb(\nu=k|\Fc_n)  = 
\\
& \frac{\sum_{\Bc \in \Pc} p_\Bc \pi_k \prod_{t=1}^k g(\Xb(t)|\Xb^{t-1}) \int_\Theta\prod_{t=k+1}^{n} f_{\Bc,\theta}(\Xb(t)|\Xb^{t-1})  \drm W(\theta)}{\sum_{j=-\infty}^\infty 
\sum_{\Bc \in \Pc} p_\Bc  \pi_j \prod_{t=1}^j g(\Xb_t|\Xb^{t-1})  \int_\Theta \prod_{t=j+1}^{n} f_{\Bc, \theta} (\Xb(t)|\Xb^{t-1}) \drm W(\theta) }
\\
& =  \frac{\pi_k \Lambda_{\pb,W}(k,n)}{q \Lambda_{\pb,W}(0,n) + \sum_{j=0}^{n-1} \pi_j\Lambda_{\pb,W}(j,n) + \Pb(\nu\ge n)}.
\end{align*}
It follows that 
\begin{align}  \label{Probnu>n}
\Pb(\nu \ge n | \Fc_n) &=  \frac{\sum_{k=n}^\infty  \pi_k \Lambda_{\pb,W}(k,n)}{q \Lambda_{\pb,W}(0,n) + \sum_{j=0}^{n-1} \pi_j \Lambda_{\pb,W}(j,n) + \Pb(\nu\ge n)}
\nonumber
\\
& =  \frac{\Pb(\nu\ge n)}{q\Lambda_{\pb,W}(0,n) + \sum_{j=0}^{n-1} \pi_j \Lambda_{\pb,W}(j,n) + \Pb(\nu\ge n)}
\nonumber
\\
& = \frac{1}{S_{\pb,W}^\pi(n) + 1}.
\end{align}
By the definition of the stopping time $T_A^W$, $S_{\pb,W}^\pi(T_A^W) \ge A$ on $\{T_A^W<\infty\}$ and $\PFA(T_A^W) = \Eb^\pi_{\Bc,\theta} [\Pb(T_A^W \le \nu | \Fc_{T_A^W}); T_A^W < \infty]$, so that
\[
\PFA(T_A^W) = \Eb^\pi_{\Bc,\theta}\brcs{(1+ S_{\pb,W}^\pi(T_A^W))^{-1}; T_A^W< \infty} \le 1/(1+A),
\]
which completes the proof of inequality \eqref{PFADMSineq}. 
\endproof

The following proposition, whose proof is given in the Appendix, provides asymptotic operating characteristics of the double-mixture procedure $T_A^W$ for large values of threshold $A$. 

\begin{proposition}\label{Lem:AOCDMS} 
Let the prior distribution of the change point belong to class $\Cb(\mu)$.   Let $r\ge 1$ and  assume that for some $0<I_{\Bc,\theta}<\infty$, $\Bc\in\Pc$, $\theta\in\Theta$, right-tail and left-tail conditions
$\C_{1}$ and $\C_{2}$ are satisfied.  

\noindent {\bf (i)} If condition $\mb{CP} {\mb 3}$ holds, then, for all $0<m \le r$ and all $\Bc\in\Pc$, $\theta\in\Theta$ as $A\to \infty$
\begin{equation} \label{MomentskDMS}
\lim_{A\to \infty} \frac{\Rc^m_{k,\Bc,\theta}(T_A^W)}{|\log A|^m} =  \frac{1}{(I_{\Bc,\theta}+\mu)^m} \quad \text{for all}~ k \in \Zbb_+.
\end{equation}

\noindent {\bf (ii)} If condition $\mb{CP} {\mb 2}$ holds, then, for all $0<m \le r$ and all $\Bc\in\Pc$, $\theta\in\Theta$ as $A\to \infty$
\begin{equation} \label{MomentsDMS}
\lim_{A\to \infty} \frac{\Rca^m_{\Bc,\theta}(T_A^W)}{|\log A|^m} =  \frac{1}{(I_{\Bc,\theta}+\mu)^m} .
\end{equation}
\end{proposition}

The following theorem shows that the double-mixture detection procedure $T_A^W$ attains the asymptotic lower bounds \eqref{LBkinclass}--\eqref{LBinclass} in Theorem~\ref{Th:LB} 
for the moments of the detection delay under conditions postulated in Proposition~\ref{Lem:AOCDMS}, 
being therefore first-order asymptotically optimal in class $\class$ as $\alpha\to0$ in the general non-i.i.d.\ case.   

\begin{theorem}\label{Th:FOAODMS} 
Let the prior distribution of the change point belong to class $\Cb(\mu)$.   Let $r\ge 1$  and assume that for some $0<I_{\Bc,\theta}<\infty$, $\Bc\in\Pc$, $\theta\in\Theta$, right-tail and left-tail conditions
$\C_{1}$ and $\C_{2}$ are satisfied. Assume that  $A=A_\alpha$ is so selected that $\PFA(T_{A_\alpha}^W) \le \alpha$ and $\log A_\alpha\sim |\log\alpha|$ as $\alpha\to0$, 
in particular $A_\alpha=(1-\alpha)/\alpha$.

\noindent {\bf (i)} If condition $\mb{CP} {\mb 3}$ holds, then $T_{A_\alpha}^W$ is first-order asymptotically optimal as $\alpha\to0$ in class $\class$, 
minimizing conditional moments of the detection delay $\Rc^m_{k,\Bc,\theta}(T)$  up to order $r$, i.e., for all $0<m \le r$ and all $\Bc\in\Pc$, $\theta\in\Theta$ as $\alpha\to0$ 
\begin{equation}\label{FOAODMSmomentsk}
 \inf_{T \in \class} \Rc^m_{k,\Bc,\theta}(T)  \sim \brc{\frac{|\log\alpha|}{I_{\Bc,\theta}+\mu}}^m \sim \Rc^m_{k,\Bc,\theta}(T_{A_\alpha}^W) \quad \text{for all}~ k \in \Zbb_+.
 \end{equation}

\noindent {\bf (ii)} If condition $\mb{CP} {\mb 2}$ holds, then $T_{A_\alpha}^W$ is first-order asymptotically optimal as $\alpha\to0$ in class $\class$, 
minimizing moments of the detection delay $\Rca^m_{\Bc,\theta}(T)$  up to order $r$, i.e., for all $0<m \le r$ and all $\Bc\in\Pc$, $\theta\in\Theta$ as $\alpha\to0$ 
\begin{equation}\label{FOAODMSmoments}
 \inf_{T \in \class} \Rca^m_{\Bc,\theta}(T)    \sim \brc{\frac{|\log\alpha|}{I_{\Bc,\theta}+\mu}}^m \sim \Rca^m_{\Bc,\theta}(T_{A_\alpha}^W) .
\end{equation}
\end{theorem}

\proof
The proof of part (i). If threshold $A_\alpha$ is so selected that  $\log A_\alpha\sim |\log\alpha|$ as $\alpha\to0$, it follows from Proposition~\ref{Lem:AOCDMS}(i) that for  $0<m \le r$ and all 
$\Bc\in\Pc$, $\theta\in\Theta$ as $\alpha\to0$
\begin{align*}
\Rc^m_{k,\Bc,\theta}(T_{A_\alpha}^W)  & \sim \brc{\frac{|\log \alpha|}{I_{\Bc,\theta}+\mu}}^m \quad \text{for all}~ k \in \Zbb_+ .
\end{align*}
This asymptotic approximation shows that the asymptotic lower bound \eqref{LBkinclass} in Theorem~\ref{Th:LB}  is attained by $T^W_{A_\alpha}$, proving the assertion (i) of the theorem.

The proof of part (ii). If threshold $A_\alpha$ is so selected that  $\log A_\alpha\sim |\log\alpha|$ as $\alpha\to0$, it follows from Proposition~\ref{Lem:AOCDMS}(ii) that for  $0<m \le r$ and all 
$\Bc\in\Pc$, $\theta\in\Theta$ as $\alpha\to0$
\begin{align*}
\Rca^m_{\Bc,\theta}(T_{A_\alpha})  & \sim \brc{\frac{|\log \alpha|}{I_{\Bc,\theta}+\mu}}^m .
\end{align*}
This asymptotic approximation along with the asymptotic lower bound  \eqref{LBinclass} in Theorem~\ref{Th:LB}  proves the assertion (ii) of the theorem.
\endproof

\subsection{Asymptotic optimality of the double-mixture procedure $\wtT_A^W$}\label{ssec:AODMSR}

Consider now the double-mixture detection procedure $\wtT_A^W$ defined in \eqref{DMSR_stat} and \eqref{DMSR_def}.

Note that 
\[
\Eb_\infty[R_{\pb,W}(n)| \Fc_{n-1}] =\sum_{\Bc\in\Pc} p_\Bc\int_\Theta \mrm{d} W(\theta) +\sum_{\Bc\in\Pc} p_\Bc \int_\Theta R_{\Bc,\theta}(n-1) \,  \mrm{d} W(\theta) = 1+ R_{\pb,W}(n-1), 
\]
and hence,  $\{R_{\pb,W}(n)\}_{n\ge 1}$ is a $(\Pb_\infty,\Fc_n)-$submartingale  with mean  $\Eb_\infty [R_{\pb,W}(n)=\omega+n$.  Therefore, by Doob's submartingale inequality, 
\begin{equation}\label{PFADoobineqDM}
 \Pb_\infty(\wtT_A^W \le k) \le (\omega+k)/A, \quad k=1,2,\dots ,
 \end{equation}
which implies the following lemma that establishes an upper bound for the PFA of the procedure $\wtT_A^W$.

\begin{lemma}\label{Lem:PFADSR} 
For all $A >0$ and any prior distribution of $\nu$ with finite mean $\bar\nu=\sum_{k=1}^\infty k \pi_k$, the PFA of the procedure $\wtT_A^W$ satisfies the inequality 
\begin{equation}\label{PFADSR}
\PFA(\wtT_A^W)  \le \frac{\omega b + \bar{\nu}} {A} ,
\end{equation}
where $b=\sum_{k=1}^\infty \pi_k$, so that if $A= A_\alpha =   (\omega b+\bar\nu)/\alpha$,
then $\PFA(\wtT_{A_\alpha}^W) \le \alpha$, i.e., $\wtT_{A_\alpha}^W \in \class$.  
\end{lemma}

As before, the prior distribution may depend on the PFA $\alpha$, so the mean $\bar\nu_\alpha = \sum_{k=0}^\infty k \, \pi_k^\alpha$ depends on $\alpha$. We also suppose that in general the head-start  
$\omega=\omega_\alpha$ depends on $\alpha$ and may go to infinity as $\alpha\to0$. Throughout this subsection we assume  that $\omega_\alpha\to\infty$ and $\bar\nu_\alpha\to\infty$ 
with such rate that the following condition holds:
\begin{equation}\label{Prior4}
\lim_{\alpha\to 0} \frac{{\log (\omega_\alpha+\bar\nu_\alpha)}}{|\log \alpha|} = 0.
\end{equation} 

The following proposition, whose proof is given to the Appendix, establishes asymptotic operating characteristics of the procedure $\wtT_A^W$ for large $A$.

\begin{proposition}\label{Lem:AOCDMSR} 
Suppose that condition \eqref{Prior4} holds and there exist positive and finite numbers $I_{\Bc,\theta}$, $\Bc\in\Pc$, $\theta\in\Theta$, such that right-tail and left-tail conditions
$\C_{1}$ and $\C_{2}$ are satisfied. Then, for all $0<m \le r$, $\Bc\in\Pc$, and $\theta\in\Theta$
\begin{equation} \label{MomentskDMSR}
\lim_{A\to\infty}  \frac{\Rc^m_{k,\Bc,\theta}(\wtT_A^W) }{(\log A)^m} =\frac{1}{I_{\Bc,\theta}^m} \quad  \text{for all}~ k \in \Zbb_+
\end{equation}
and
\begin{equation} \label{MomentsDMSR}
\lim_{A\to\infty}  \frac{\Rca^m_{\Bc,\theta}(\wtT_A^W)}{(\log A)^m} =\frac{1}{I_{\Bc,\theta}^m} .
\end{equation}
\end{proposition}

Using asymptotic approximations \eqref{MomentskDMSR}--\eqref{MomentsDMSR} in Proposition~\ref{Lem:AOCDMSR}, we now can easily prove that the double-mixture procedure $\wtT_A^W$ attains 
the asymptotic lower bounds \eqref{LBinclass}--\eqref{LBkinclass} in Theorem~\ref{Th:LB} for moments of the detection delay when $\mu=0$. This means that the procedure $\wtT_{A_\alpha}^W$ 
is first-order asymptotically optimal as  $\alpha\to0$  if the prior distribution of the change point belongs to class $\Cb(\mu=0)$.

\begin{theorem}\label{Th:AoptDSR} 
Assume  that the head-start $\omega=\omega_\alpha$ of the statistic $R_{\pb,W}(n)$ and the mean value $\bar\nu=\bar\nu_\alpha$ of the prior distribution $\{\pi_k^\alpha\}$ approach infinity 
as $\alpha\to0$  with such rate that condition \eqref{Prior4} is satisfied.
Suppose further that for some $0<I_{\Bc,\theta}<\infty$ and $r \ge 1$ conditions $\C_1$ and $\C_2$ are satisfied.  
If threshold $A_\alpha$ is so selected that $\PFA(\wtT_{A_\alpha}^W) \le \alpha$  and $\log A_\alpha \sim |\log \alpha|$ as $\alpha \to 0$, in particular 
$A_\alpha=(\omega_\alpha b_\alpha+\bar\nu_\alpha)/\alpha$, then for all $0<m \le r$, $\Bc\in\Pc$, and $\theta\in\Theta$ as $\alpha\to 0$
\begin{equation} \label{MomentsDSRmu0}
\begin{split}
\Rca^m_{\Bc,\theta}(\wtT_{A_\alpha}^W) & \sim \brc{\frac{|\log\alpha|}{I_{\Bc,\theta}}}^m ,
\\
\Rc^m_{k,\Bc,\theta}(\wtT_{A_\alpha}^W) & \sim \brc{\frac{|\log\alpha|}{I_{\Bc,\theta}}}^m  \quad \text{for all} ~ k\in \Zbb_+.
\end{split}
\end{equation}
If the prior distribution belongs to class $\Cb(\mu=0)$, then for all $0<m \le r$, $\Bc\in\Pc$, and $\theta\in\Theta$ as $\alpha\to 0$
\begin{equation} \label{MomentskMSRmu0opt}
\begin{split}
\inf_{T \in \class} \Rca^m_{\Bc,\theta}(T) & \sim \Rca^m_{\Bc,\theta}(\wtT_{A_\alpha}^W) ,
\\
\inf_{T \in \class} \Rc^m_{k,\Bc,\theta}(T) & \sim \Rc^m_{k,\Bc,\theta}(\wtT_{A_\alpha}^W) \quad \text{for all} ~ k\in \Zbb_+.
\end{split}
\end{equation}
Therefore, the procedure $\wtT_{A_\alpha}^W$ is asymptotically optimal as $\alpha\to0$ in class $\class$, minimizing moments of the detection delay up to order $r$, if the prior distribution of the 
change point belongs to class $\Cb(\mu)$ with $\mu=0$. 
\end{theorem}

\proof
If $A_\alpha$ is so selected that  $\log A_\alpha\sim |\log\alpha|$ as $\alpha\to0$, then asymptotic approximations \eqref{MomentsDSRmu0} follow immediately 
from asymptotic approximations \eqref{MomentskDMSR}--\eqref{MomentsDMSR} in Proposition~\ref{Lem:AOCDMSR}.  Since these approximations are the same as the asymptotic lower bounds 
\eqref{LBinclass}--\eqref{LBkinclass} in Theorem~\ref{Th:LB} for $\mu=0$, this shows that these bounds are attained by 
the detection procedure $\wtT_{A_\alpha}^W$ whenever the prior distribution belongs to class $\Cb(\mu=0)$, which completes the proof of assertions \eqref{MomentskMSRmu0opt}. 
\endproof

While the procedure  $\wtT_{A_\alpha}^W$ is asymptotically optimal for heavy-tailed priors (when $\mu=0$ in \eqref{Prior}),  comparing \eqref{MomentsDSRmu0}
with the assertion of Theorem~\ref{Th:FOAODMS} (see \eqref{FOAODMSmomentsk} and \eqref{FOAODMSmoments}) we can see that the procedure $\wtT_{A_\alpha}^W$ is not asymptotically optimal 
when $\mu>0$, i.e.,  for the priors with asymptotic exponential tails. This can be expected since the statistic $R_{\pb,W}(n)$ uses the uniform prior distribution of the change point. 

\section{Asymptotic optimality with respect to the average risk}\label{sec:AOIR}

In this section, instead of the constrained optimization problem \eqref{sec:PrbfBayes} we are interested in the unconstrained Bayes problem \eqref{FOAOdefAvRisk} 
with the average (integrated) risk function $\rho_{\pi, \pb, W}^{c,r}(T)$ defined in
\eqref{Averrisk}, where $c>0$ is the cost of delay per time unit  and  $r \ge 1$. 
Below we show that the double-mixture procedure $T_A^W$ with a certain threshold $A=A_{c,r}$ that depends on the cost $c$ is asymptotically optimal, minimizing the average 
risk $\rho_{\pi, \pb, W}^{c,r}(T)$  to first order over all stopping times as the cost vanishes, $c\to0$. 

Recall that $\Rca_{\pb, W}^{r}(T)$, $D_{\mu,r}$, and $G_{c,r}(A)$ are defined in \eqref{R}, \eqref{Dr}, and \eqref{Gfunction}, respectively. 
Since $\PFA(T_A^W) \approx 1/(1+A)$ (ignoring an excess over the boundary), using the asymptotic formula~\eqref{MomentsDMS} we obtain that for a large $A$
\[
\Rca_{\pb,W}^r(T_A) \approx \sum_{\Bc\in\Pc} p_\Bc \int_\Theta \brc{\frac{\log A}{I_{\Bc,\theta} + \mu}}^r \, \drm W(\theta) = (\log A)^r D_{\mu,r} .
\]
So for large~$A$ the average risk of the procedure $T_A$ is approximately equal to
\begin{align*}
\rho_{\pi, \pb,W}^{c,r}(T_A^W)  = \PFA(T_A^W) + c \, [1-\PFA(T_A^W)] \Rca_{\pb,W}^r(T_A^W)
\approx G_{c,r}(A).
\end{align*}
The procedure $T_{A_{c,r}}$ with the threshold value $A=A_{c,r}$ that minimizes $G_{c,r}(A)$, $A>0$, which 
is a solution of the equation \eqref{eq:threshold} (see below), is a reasonable candidate for being asymptotically optimal in the Bayesian sense as $c\to0$, i.e., in the asymptotic problem  
\eqref{FOAOdefAvRisk}.
The next theorem shows that this is true under conditions $\C_1$ and $\C_2$ when the set $\Theta$ is compact and that the same is true for the procedure $\wtT_{A_{c,r}}^W$ 
with certain threshold $A_{c,r}$ in class of priors $\C(\mu)$ with $\mu=0$.

\begin{theorem} \label{Th:FOasopt_pureBayes}  
 Assume that for some $0<I_{\Bc,\theta}<\infty$, $\Bc\in\Pc$, $\theta\in\Theta$, right-tail and left-tail conditions
$\C_{1}$ and $\C_{2}$ are satisfied and that $\Theta$ is a compact set.  

\noindent {\bf (i)}   If the prior distribution of the change point $\pi^c=\{\pi_k^c\}$  satisfies condition \eqref{Prior} with 
$\lim_{c\to0} \mu_c =\mu$ and 
\begin{equation}\label{Prior3c}
\lim_{c\to 0} \frac{{\sum_{k=0}^\infty  \pi_k^c |\log \pi_k^c|^r}}{|\log c|^r} = 0, 
\end{equation} 
and if threshold $A=A_{c,r}$ of the procedure $T_A^W$ is the solution of the equation
\begin{equation} \label{eq:threshold}
r D_{\mu,r} A (\log A)^{r-1}=1/c,
\end{equation}
 then
\begin{equation} \label{RiskAOgenexpflat}
\inf_{T\ge 0} \rho_{\pi, \pb,W}^{c,r}(T) \sim D_{\mu,r} \, c \, |\log c|^r  \sim \rho_{\pi, \pb,W}^{c,r}(T_{A_{c,r}}) \quad \text{as}~~c\to 0,
\end{equation}
i.e., $T_{A_{c,r}}^W$ is first-order asymptotically Bayes as $c\to0$.

\noindent {\bf (ii)} If the head-start $\omega_c$ and the mean of the prior distribution $\bar\nu_c$ approach infinity at such rate that 
\begin{equation}\label{Prior4c}
\lim_{c \to 0} \frac{{\log (\omega_c+\bar\nu_c)}}{|\log c|} = 0
\end{equation}
and if threshold $A=A_{c,r}$ of the procedure $\wtT_{A}^W$  is the solution of the equation 
\begin{equation} \label{threshold2}
r D_{0,r} A (\log A)^{r-1}=(\omega_c b_c +\bar{\nu}_c)/c,
\end{equation}
then
\begin{equation} \label{RiskDMSRBayes}
\begin{split}
 \rho_{\pi, \pb,W}^{c,r}(\wtT_{A_{c,r}})  \sim D_{0,r} \, c \, |\log c|^r   \quad \text{as}~~c\to 0,
\end{split}
\end{equation}
i.e., $\wtT_{A_{c,r}}^W$ is first-order asymptotically Bayes as $c\to0$ in the class of priors $\C(\mu=0)$.
\end{theorem}

\begin{proof}
The proof is based on the technique used by Tartakovsky for proving Theorems~5 and 6 in \cite{TartakovskyIEEEIT2018} for the single stream problem with an unknown 
post-change parameter for the prior with
the fixed $\mu_c=\mu$ for all $c$ in condition \eqref{Prior} and  with positive $\mu_c$ which vanishes when $c\to0$. A more general prior considered in this article is 
handled analogously. 

The proof of part (i). Since $\Theta$ is compact it follows from Proposition~\ref{Lem:AOCDMS}(ii) (cf.\ the asymptotic approximation \eqref{MomentsDMS}) that under conditions $\C_{1}$, $\C_{2}$,
and \eqref{Prior3c} as $A\to\infty$
\[
\begin{split}
\Rca_{\pb,W}^r(T_A) & \sim \sum_{\Bc\in\Pc} p_\Bc \int_\Theta \brc{\frac{\log A}{I_{\Bc,\theta} + \mu}}^r \, \drm W(\theta) 
\\
&  = D_{\mu,r} (\log A)^r .
\end{split}
\]
By Lemma~\ref{Lem:PFADMS}, $\PFA(T_A^W) \le 1/(A+1)$. Obviously, if threshold $A_{c,r}$ satisfies equation \eqref{eq:threshold}, then $\log A_{c,r} \sim |\log c|$ and 
$\PFA(T_{A_{c,r}}^W) =o(c|\log c|^r)$ as $c\to 0$ (for any $r \ge 1$ and $\mu \ge 0$). As a result, 
\[
\rho_{\pi, \pb,W}^{c,r}(T_{A_{c,r}}) \sim D_{\mu,r} \, c \, |\log c|^r  \quad \text{as}~~c\to 0,
\]
which along with the lower bound \eqref{LBAR} in Theorem~\ref{Th:LB} completes the proof of the assertion (i).

The proof of part (ii). Since $\Theta$ is compact it follows from the asymptotic approximation \eqref{MomentsDMSR} in Proposition~\ref{Lem:AOCDMSR}  that under conditions $\C_{1}$, $\C_{2}$,
and \eqref{Prior4c} as $A\to\infty$
\[
\begin{split}
\Rca_{\pb,W}^r(\wtT_A^W) & \sim \sum_{\Bc\in\Pc} p_\Bc \int_\Theta \brc{\frac{\log A}{I_{\Bc,\theta}}}^r \, \drm W(\theta) 
\\
&  = D_{0,r} (\log A)^r .
\end{split}
\]
Define 
\[
\widetilde{G}_{c,r}(A)=(\omega_c b_c +\bar{\nu}_c)/A + c \, D_{0,r} (\log A)^r.
\]
By Lemma~\ref{Lem:PFADSR}, $\PFA(\wtT_A) \le (\omega_c b_c +\bar{\nu}_c)/A$, so that for a sufficiently large $A$,
\[
\rho_{\pi, \pb,W}^{c,r}(\wtT_{A}^W) \approx \widetilde{G}_{c,r}(A).
\] 
Threshold $A_{c,r}$, which satisfies equation \eqref{threshold2}, minimizes $\widetilde{G}_{c,r}(A)$. By assumption \eqref{Prior4c},  $\omega_c b_c +\bar{\nu}_c = o(|\log c|)$ as $c\to0$, so
that $\log A_{c,r} \sim |\log c|$ and $\PFA(\wtT_{A_{c,r}}) \le (\omega_c b_c+\bar\nu_c)/A_{c,r} = o(c|\log c|^r)$ as $c\to 0$. Hence, it follows that
\[
\rho_{\pi, \pb,W}^{c,r}(\wtT_{A_{c,r}}^W)  \sim \widetilde{G}_{c,r}(A_{c,r}) \sim D_{0,r} \, c \, |\log c|^r  \quad \text{as}~~c\to 0.
\]
This implies the asymptotic approximation~\eqref{RiskDMSRBayes}. Asymptotic optimality of $\wtT_{A_{c,r}}^W$ in the class of priors $\C(\mu=0)$ follows from \eqref{RiskAOgenexpflat} and
\eqref{RiskDMSRBayes}.
\end{proof}

\section{A remark on asymptotic optimality for a putative value of the post-change parameter}\label{sec:AOMS} 

If the value of the post-change parameter $\theta=\vartheta$ is known or its putative value $\vartheta$ is of special interest,
representing a nominal change, then it is reasonable to turn the double-mixture procedures 
$T_A^W$ and $\wtT^W$ in single-mixture procedures $T_A^\vartheta$ and $\wtT_A^\vartheta$ by taking the degenerate weight function $W$ concentrated at $\vartheta$.  
These procedures are of the form
\begin{equation*}
T_A^\vartheta=\inf\set{n \ge 1: S_{\pb,\vartheta}^\pi(n) \ge A}, \quad
\wtT_A^\vartheta=\inf\set{n \ge 1: R_{\pb,\vartheta}(n) \ge A},
\end{equation*}
and they have first-order asymptotic optimality properties at the point $\theta=\vartheta$ (and only at this point) with respect to $\Rc_{k,\Bc,\vartheta}(T)$ and $\Rca_{\Bc,\vartheta}(T)$
 when the right-tail condition $\C_1$ is satisfied for $\theta=\vartheta$ and the following left-tail condition holds:
 
\vspace{2mm}
\noindent $\widetilde{\C}_{2}$. {\em  For every $\Bc\in\Pc$, $\varepsilon>0$, and for some $r\ge 1$}
\begin{equation}\label{rcompunifLeft}
\sum_{n=1}^\infty \, n^{r-1} \, \sup_{k \in \Zbb_+} \Pb_{k,\Bc,\vartheta}\brc{\frac{1}{n}\lambda_{\Bc,\vartheta}(k,k+n) < I_{\Bc,\vartheta}  - \varepsilon} <\infty  .
\end{equation}

The assertions of Theorem~\ref{Th:FOasopt_pureBayes} also hold under conditions $\C_1$ and $\widetilde{\C}_2$ for the average risk 
\[
\rho_{\pi, \pb, \vartheta}^{c,r}(T) =  \PFA(T) + c \, \sum_{\Bc\in\Pc} p_{\Bc}  \Eb_{\Bc,\vartheta}^\pi[(T-\nu)^+]^r
\]
with
\[
D_{\mu,r} =  \sum_{\Bc\in\Pc} p_\Bc \brc{\frac{1}{I_{\Bc,\vartheta} + \mu}}^r .
\]

\section{Asymptotic optimality in the case of independent streams}\label{sec:indstreams}

A particular, still very general scenario is where the data streams are mutually independent (but still have a quite general statistical structure) is of special interest for many applications. 
In this case, the model is given by \eqref{ind} and, as discussed in Subsection~\ref{sssec:indepchannels}, the implementation of detection procedures may be feasible since the LR process 
$\Lambda_{\pb,\theta}(k,n)$ can be easily computed (see \eqref{mixLRind}). Moreover, in the case of independent data streams all the results obviously hold for different 
values of the parameter $\theta=\theta_i\in\Theta_i$ in streams, which we will assume in this section. Specifically, we will write $\theta_i$ for a post-change parameter in the $i$th stream and
bold $\teb_\Bc=(\theta_i, i\in \Bc)\in \mb{\Theta}_\Bc$ for the vector of post-change parameters in the subset of streams $\Bc$.

Since the data are independent across streams, for an assumed
value of the change point $\nu = k$,  stream $i \in [N]$, and the post-change parameter value in the $i$th stream $\theta_i$, the LLR of observations accumulated by time $k+n$ is given by
\[
\lambda_{i,\theta_i}(k, k+n)= \sum_{t=k+1}^{k+n}\log \frac{f_{i,\theta_i}(X_i(t)| \Xb_i^{t-1})}{g_{i}(X_i(t)| \Xb_i^{t-1})}, \quad n \ge 1.
\]
Let
\begin{align*}
\beta_{M,k}(\varepsilon,i,\theta_i) & =\Pb_{k,i,\theta_i}\set{\frac{1}{M}\max_{1 \le n \le M} \lambda_{i,\theta_i}(k, k+n) \ge (1+\varepsilon) I_{i,\theta_i}},
\\
U_r(\varepsilon, i,\theta_i) & =   \sum_{n=1}^\infty n^{r-1} \sup_{k\in \Zbb_+} \Pb_{k,i, \theta_i}\set{\frac{1}{n} \inf_{\vartheta\in\Gamma_{\delta,\theta_i}}\lambda_{i,\vartheta}(k, k+n) < I_{i,\theta_i}  - \varepsilon}.
\end{align*}

Assume that the following conditions are satisfied for local statistics in data streams:
\vspace{2mm}

\noindent $\C_{1}^{(i)}$. {\em  There exist positive and finite numbers $I_{i,\theta_i}$, $\theta_i\in \Theta_i$, $i \in [N]$, such that for any  $\varepsilon >0$}
\begin{equation}\label{sec:Pmaxi}
\lim_{M\to\infty} \beta_{M,k}(\varepsilon,i,\theta_i) =0 \quad \text{for all}~ k\in \Zbb_+, \theta_i\in\Theta_i,  i \in [N]  ;
\end{equation}

\noindent $\C_{2}^{(i)}$. {\em  For any $\varepsilon>0$ there exists $\delta=\delta_{\varepsilon}>0$ such that $W(\Gamma_{\delta,\theta_i})>0$
and for any $\varepsilon>0$ and some $r\ge 1$}
\begin{equation}\label{rcompLefti}
U_r(\varepsilon, i,\theta_i) < \infty \quad \text{for all}~  \theta_i\in\Theta_i, i\in [N]  .
\end{equation}

Let $I_{\Bc,\teb_{\Bc}} = \sum_{i\in\Bc} I_{i,\theta_i}$. Since the LLR process $\lambda_{\Bc,\teb_\Bc}(k,k+n)$ is the sum of independent local LLRs,
$\lambda_{\Bc,\teb_\Bc}(k,k+n) = \sum_{i\in\Bc} \lambda_{i,\theta_i}(k,k+n)$ (see \eqref{LRind}), it is easy to show that
\[
\beta_{M,k}(\varepsilon,\Bc,\teb_\Bc) \le \sum_{i\in\Bc} \beta_{M,k}(\varepsilon,i,\theta_i) ,
\]
so that local conditions $\C_{1}^{(i)}$ imply global right-tail condition $\C_{1}$. This is true, in particular, if the normalized local LLRs $n^{-1}\lambda_{i,\theta_i}(k, k+n)$ 
converge $\Pb_{k,i, \theta_i}$-a.s.\ to  $I_{i,\theta_i}$,  $i=1,\dots,N$, in which case the SLLN for the global LLR \eqref{sec:MaRe.1} holds with $I_{\Bc,\theta_{\Bc}} = \sum_{i\in\Bc} I_{i,\theta_i}$.
Also,
\[
U_r(\varepsilon, \Bc,\teb_\Bc) \le  \sum_{i\in\Bc} U_r(\varepsilon, \Bc,\theta_i) ,
\]
which shows that local left-tail conditions $\C_2^{(i)}$ imply global left-tail condition $\C_2$.

Thus, Theorem~\ref{Th:FOAODMS} and Theorem~\ref{Th:AoptDSR} imply the following results on asymptotic properties of the double-mixture procedures $T_A^W$ and $\wtT_A^W$.

\begin{corollary}\label{Cor:Cor1}
Let $r\ge 1$  and assume that for some positive and finite numbers $I_{i,\theta_i}$, $\theta_i\in\Theta_i$, $i=1\dots,N$, 
right-tail and left-tail conditions $\C_{1}^{(i)}$ and $\C_{2}^{(i)}$ for local data streams are satisfied. 

\noindent {\bf (i)} Let the prior distribution of the change point belong to class $\Cb(\mu)$.  
If $A=A_\alpha$ is so selected that $\PFA(T_{A_\alpha}^W) \le \alpha$ and $\log A_\alpha\sim |\log\alpha|$ as $\alpha\to0$, 
in particular $A_\alpha=(1-\alpha)/\alpha$, and if conditions  $\mb{CP} {\mb 2}$ and $\mb{CP} {\mb 3}$ are satisfied, then 
asymptotic formulas \eqref{FOAODMSmomentsk} and \eqref{FOAODMSmoments} hold with $I_{\Bc,\theta} = I_{\Bc,\teb_\Bc}=\sum_{i\in\Bc} I_{i,\theta_i}$, and therefore,  
$T_{A_\alpha}^W$ is first-order asymptotically optimal as $\alpha\to0$ in class $\class$, 
minimizing moments of the detection delay  up to order $r$ uniformly for all $\Bc\in\Pc$ and  $\teb_\Bc \in \mb{\Theta}_\Bc$.

\noindent {\bf (ii)} If threshold $A_\alpha$ is so selected that $\PFA(\wtT_{A_\alpha}^W) \le \alpha$  and $\log A_\alpha \sim |\log \alpha|$ as $\alpha \to 0$, in particular 
$A_\alpha=(\omega_\alpha b_\alpha+\bar\nu_\alpha)/\alpha$, and if condition \eqref{Prior4} is satisfied, then asymptotic formulas \eqref{MomentsDSRmu0} and \eqref{MomentskMSRmu0opt} hold
with $I_{\Bc,\theta} = I_{\Bc,\teb_\Bc}=\sum_{i\in\Bc} I_{i,\theta_i}$, and therefore,  $\wtT_{A_\alpha}^W$ is first-order asymptotically optimal as $\alpha\to0$ in class $\class$, 
minimizing moments of the detection delay  up to order $r$ uniformly for all $\Bc\in\Pc$ and  $\teb_\Bc\in\mb{\Theta}_\Bc$ if the prior distribution of the 
change point belongs to class $\Cb(\mu)$ with $\mu=0$. 
\end{corollary}

\begin{remark} \label{Rem: SufCond}
Obviously, the following condition implies condition $\C_2^{(i)}$:

\noindent $\C_{3}^{(i)}$. {\em  For any $\varepsilon>0$ there exists $\delta=\delta_{\varepsilon}>0$ such that $W(\Gamma_{\delta,\theta_i}) >0$. Let the $\Theta_i \to\bbr_{+}$  functions
$I_{i}(\theta_i)=I_{i,\theta_i}$ be continuous and assume that  for every compact set $\Theta_{c,i}\subseteq \Theta_i$, every  $\varepsilon>0$, and some $r\ge 1$}
\begin{equation}\label{rcompLeft*}
\Upsilon^{*}_r(\varepsilon, i, \Theta_{c,i}):= \sup_{\theta\in\Theta_{c,i}} \Upsilon_r(\varepsilon,i,\theta_i)  <\infty \quad \text{for all}~ i \in [N].
\end{equation} 
Note also that if there exists continuous $\Theta_i\times \Theta_i \to\bbr_{+}$ functions $I_i(\vartheta_i,\theta_i)$ 
such that for any  $\varepsilon>0$, any compact  $\Theta_{c,i} \subseteq \Theta_i$ and some $r\ge 1$
\begin{equation}\label{rcompSup}
\begin{split}
\Upsilon^{**}_r(\varepsilon, \Theta_{c,i})& :=
\sum_{n=1}^\infty \, n^{r-1} \, \sup_{k \in \Zbb_+} \,\sup_{\theta_i\in \Theta_{c,i}} \Pb_{k,i,\theta_i}\brc{ \sup_{\vartheta_i \in\Theta_{c,i}} \left\vert\frac{1}{n}\lambda_{i,\vartheta_i}(k, k+n) - 
I_i(\vartheta_i,\theta_i)\right\vert > \varepsilon} <\infty
\\
& \quad \text{for all}~ i \in [N],
\end{split}
\end{equation}
then conditions $\C_3^{(i)}$, and hence, conditions $\C_2^{(i)}$ are satisfied with $I_{\Bc,\teb_\Bc}= \sum_{i\in\Bc} I_i(\theta_i,\theta_i)$ since
\begin{equation*}
\Pb_{k,i,\theta_i}\brc{\frac{1}{n} \inf_{|\vartheta_i - \theta_i|<\delta} \lambda_{i,\vartheta_i}(k,k+n) < I_{i,\theta_i}  - \varepsilon} \le 
\Pb_{k,i,\theta_i}\brc{\sup_{\vartheta_i \in\Theta_{c,i}} \left\vert\frac{1}{n}\lambda_{i,\vartheta_i}(k, k+n) - I_i(\vartheta_i,\theta_i)\right\vert > \varepsilon} .
\end{equation*}        
\end{remark}

Conditions $\C_3^{(i)}$ and \eqref{rcompSup} are useful for establishing asymptotic optimality of proposed detection procedures in particular examples.

\section{Examples}\label{sec:Ex}

\subsection{Detection of signals with unknown amplitudes in a multichannel system}\label{ssec:Ex1}
In this subsection, we consider the $N$-channel quickest detection problem, which is 
an interesting real-world example, arising in multichannel radar systems and electro-optic imaging systems where it is required to detect an unknown number of randomly appearing signals from
objects in clutter and noise (cf., e.g., \cite{TNB_book2014,Tartakovsky&Brown-IEEEAES08, Bakutetal-book63}).

Specifically,  
we are interested in the quickest detection of deterministic signals $\theta_i S_{i,n}$ with unknown amplitudes $\theta_i>0$ that appear at an unknown time $\nu$
in additive noises $\xi_{i,n}$ in an $N$-channel system, i.e., observations in the $i$th channel have the  form  
$$
X_{i,n}=\theta_i S_{i,n} \Ind{n > \nu}  +\xi_{i,n},\quad n \ge 1, ~ i =1,\dots,N .
$$
Assume that mutually independent noise processes $\{\xi_{i,n}\}_{n \in\Zbb_+}$ are $p$th order Gaussian autoregressive processes AR$(p)$ that obey recursions
\begin{equation}\label{sec:Ex.1}
\xi_{i,n} = \sum_{j=1}^p \beta_{i,j} \xi_{i,n-j} + w_{i,n}, \quad n \ge 1, 
\end{equation}
where   $\{w_{i,n}\}_{n\ge 1}$, $i=1,\dots,N$, are mutually independent i.i.d.\ normal $\Nc(0,\sigma_i^2)$ sequences ($\sigma_i>0$), so the observations in channels 
$X_{1,n},\dots,X_{N,n}$ are independent of each other. The initial values $\xi_{i,1-p}, \xi_{i,2-p}, \dots, \xi_{i,0}$ are arbitrary random or deterministic numbers, 
in particular we may set zero initial conditions $\xi_{i,1-p}=\xi_{i,2-p}=\cdots=\xi_{i,0}=0.$ 
The coefficients $\beta_{i,1},\dots,\beta_{i,p}$ and variances $\sigma_i^2$ are known and all roots of the equation $z^p -\beta_{i,1} z^{p-1} - \cdots - \beta_{i,p}=0$ are in the interior of the unit circle, so that 
the AR($p$) processes are stable.  

Define the $p_n$-th order residual 
\[
\widetilde{Y}_{i,n} = Y_{i,n}- \sum_{j=1}^{p_n} \beta_{i,j} Y_{i,n-j}, \quad n \ge 1,
\]
where $p_n =p$ if $n > p$ and $p_n =n$ if $n \le p$.  It is easy to see that the conditional pre-change and post-change densities in the $i$th channel are
\begin{align*}
g_i(X_{i,n}|\Xb_i^{n-1})& = f_{0,i}(X_{i,n}|\Xb_i^{n-1})= \frac{1}{\sqrt{2\pi \sigma_i^2}} \exp\set{-\frac{\wtX_{i,n}^2}{2\sigma_i^2}},
\\
f_{\theta_i}(X_{i,n}|\Xb_i^{n-1})&=   \frac{1}{\sqrt{2\pi \sigma_i^2}} \exp\set{-\frac{(\wtX_{i,n}-\theta_i \wtS_{i,n})^2}{2\sigma_i^2}} , \quad
\theta_i \in \Theta = (0, \infty),
\end{align*}
and that for all $k \in \Zbb_+$ and $n \ge 1$  the LLR in the $i$th channel has the form
$$
\lambda_{i,\theta_i}(k, k+n) = \frac{\theta_i}{\sigma_i^2}  \sum_{j=k+1}^{k+n} \wtS_{i,j} \wtX_{i,j} -\frac{\theta_i^2 \sum_{j=k+1}^{k+n} \wtS_{i,j}^2}{2 \sigma_i^2} .
$$
Since under measure $\Pb_{k,i,\vartheta_i} $ the random variables $\{\wtX_{i,n}\}_{n\ge k+1}$ are independent Gaussian random variables $\Nc(\vartheta_i \wtS_{i,n},\sigma_i^2)$, 
under  $\Pb_{k,i,\vartheta_i} $ the LLR $\lambda_{i,\theta_i}(k, k+n) $ is a Gaussian process (with independent non-identically distributed increments) with mean and variance
\begin{equation}\label{LLRAR1}
\Eb_{k,i,\vartheta} [\lambda_{i, \theta_i, \vartheta_i}(k, k+n)]= \frac{2 \theta_i \vartheta_i -\theta_i^2}{2\sigma_i^2} \sum_{j=k+1}^{k+n} \wtS_{i,j}^2 , \quad
\Var_{k,i,\vartheta_i} [\lambda_{i, \theta_i, \vartheta_i}(k, k+n)] =   \frac{\theta_i^2}{\sigma_i^2} \sum_{j=k+1}^{k+n} \wtS_{i,j}^2 .
\end{equation}

Assume that 
\begin{equation*}
\lim_{n\to \infty} \frac{1}{n}  \sup_{k \in \Zbb_+} \sum_{j=k+1}^{k+n} \wtS_{i,j}^2 = Q_i ,
\end{equation*}
where  $0<Q_i <\infty$. This is typically the case in most signal processing applications, e.g., in radar applications where the signals $\theta_i S_{i,n}$ are the sequences of harmonic pulses. 
Then for all $k\in\Zbb_+$ and $\theta_i\in (0,\infty)$
\[
\frac{1}{n}\lambda_{i,\theta_i}(k, k+n)  \xra[n\to\infty]{ \Pb_{k,i,\theta_i} -\text{a.s.}} \frac{\theta_i^2 Q_i}{2\sigma_i^2} =I_{i,\theta_i},
\]
so that condition $\C_1^{(i)}$ holds. Furthermore, since all moments of the LLR are finite it can be shown (cf.~\cite{TartakovskyIEEEIT2018})  that condition $\widetilde{\C}_2^{(i)}$ (and hence, condition 
$\C_2^{(i)}$) holds for all $r \ge 1$. 

Thus, by Corollary~\ref{Cor:Cor1}, the double-mixture procedure $T_A^W$ minimizes as $\alpha\to0$ all positive moments of the detection delay and  asymptotic formulas
\eqref{FOAODMSmomentsk} and \eqref{FOAODMSmoments} hold with $I_{\Bc,\teb_\Bc}=\sum_{i\in\Bc}  \frac{\theta_i^2 Q_i}{2\sigma_i^2}$.

\subsection{Detection of non-additive changes in mixtures} 
\label{ssec:Ex2}

Assume that the observations across streams are independent. Let $p_{1,i}(X_{i,n})$, $p_{2,i}(X_{i,n})$, and $f_{\theta_i}(X_{i,n})$ be distinct densities, $i=1,\dots,N$. Consider an example with non-additive changes where the observations in the $i$th stream 
in the normal mode follow the pre-change joint density 
\[
g_i(\Xb_i^{n}) =   \beta_i\prod_{j=1}^n p_{1,i}(X_{i,j}) + (1-\beta_i) \prod_{j=1}^n p_{2,i}(X_{i,j}),
\]
which is the mixture density with a mixing probability $0 < \beta_i < 1$, and in the abnormal mode the observations follow the post-change joint density
\[
f_{\theta_i}(\Xb_i^n)= \prod_{j=1}^n f_{\theta_i}(X_{i,j}) , \quad \theta_i\in \Theta_i. 
\]
Therefore, the observations $\{X_{i,n}\}_{n\ge 1}$ in the $i$th stream are dependent with the conditional probability density  
\[
g_i(X_{i,n} \mid \Xb_i^{n-1})  = \frac{\beta_i\prod_{j=1}^n p_{1,i}(X_{i,j}) + (1-\beta_i) \prod_{j=1}^n p_{2,i}(X_{i,j})}{\beta_i\prod_{j=1}^{n-1} p_{1,i}(X_{i,j}) + (1-\beta_i) \prod_{j=1}^{n-1} p_{2,i}(X_{i,j})}, 
\quad \nu > n
\]
before the change occurs and i.i.d.\ with  density $f_{\theta_i}(X_{i,n})$ after the change occurs ($n \ge \nu$). Note that in contrast to the previous example, pre-change densities $g_i$ do not belong to the 
same parametric family as post-change densities $f_{\theta_i}$.

Define $\Lc_{i,n}^{(s)}(\theta_i)= \log[f_{\theta_i}(X_{i,n})/p_{s,i}(X_{i,n})]$;  $I_{\theta_i}^{(s)} = \Eb_{0,\theta_i}[\Lc_{i,1}^{(s)}(\theta_i)]$, $s=1,2$;  $\Delta G_{i,n}=p_{1,i}(X_{i,n})/p_{2,i}(X_{i,n})$; 
$G_{i,n}=\prod_{j=1}^n\Delta G_{i,j}$; and $v_i=\beta_i/(1-\beta_i)$. It is easily seen that
$$
\frac{f_{\theta_i}(X_{i,n})}{g_i(X_{i,n}\mid \Xb_i^{n-1})} = \exp\set{\Lc_{i,n}^{(2)}(\theta_i)} \frac{1 -\beta_i + \beta_i G_{i,n-1}}{1-\beta_i+ \beta_i G_{i,n}}.
$$
Observing that 
$$
\prod_{j=k+1}^{k+n} \frac{1 -\beta_i + \beta_i G_{i,j-1}}{1 -\beta_i + \beta_i G_{i,j}} =  \frac{1+ v_i G_{i,k}}{1+ v_i G_{i,k+n}},
$$
we obtain
\[
\prod_{j=k+1}^{k+n} \frac{f_{\theta_i}(X_{i,j})}{g_i(X_{i,j}\mid \Xb_i^{j-1})} = \exp\set{\sum_{j=k+1}^{k+n} \Lc_{i,j}^{(2)}(\theta_i)}  \frac{1+ v_i G_{i,k}}{1+ v_i G_{i,k+n}},
\]
and therefore,
\begin{equation} 
\label{LLRmix}
\lambda_{i,\theta_i}(k,k+n) = \sum_{j=k+1}^{k+n} \Lc_{i,j}^{(2)}(\theta_i) + \log \frac{1+ v_i G_{i,k}}{1+ v_i G_{i,k+n}}.
\end{equation}

Assume that $I_{\theta_i}^{(1)} > I_{\theta_i}^{(2)}$. Then $\Eb_{k, i,\theta_i} [\log \Delta G_{in}] = I_{\theta_i}^{(2)}-I_{\theta_i}^{(1)} <0$ for $k \le n$, and hence, for all $k \in \Zbb_+$
$$
G_{i,k+n} = G_{i,k} \prod_{j=k+1}^{k+n} \Delta G_{i,j} \xra[n\to\infty]{\Pb_{k,i,\theta_i}\mbox{-a.s.}} 0
$$
and
\[
\frac{1}{n} \log \frac{1+ v_i G_{i,k}}{1+ v_i G_{i,k+n}} \xra[n\to\infty]{\Pb_{k,i,\theta_i}\mbox{-a.s.}} 0.
\]
Since under $\Pb_{k,i,\theta_i}$  the random variables $\Lc_{i,n}^{(2)}(\theta_i)$, $n=  k+1, k+2, \dots $ are i.i.d.\  with mean $I_{\theta_i}^{(2)}$,  we have
\[
\frac{1}{n} \lambda_{i,\theta_i}(k,k+n)    \xra[n\to\infty]{\Pb_{k, i,\theta_i}-\text{a.s.}} I_{\theta_i}^{(2)},
\]
and hence, condition $\C_1^{(i)}$ holds with $I_{i,\theta_i}= I_{\theta_i}^{(2)}$.

Now, under $\Pb_{k,i,\theta_i}$ the LLR $\lambda_{i,\vartheta_i}(k,k+n)$ can be written as
\[
\lambda_{i,\vartheta_i}(k,k+n) =  \sum_{j=k+1}^{k+n} \Lc_{i,j}^{(2)}(\vartheta_i, \theta_i) + \psi_i(k,n),
\] 
where $\Lc_{i,j}^{(2)}(\vartheta_i, \theta_i)$ is the statistic $\Lc_{i,j}^{(2)}(\vartheta_i)$ under $\Pb_{k,i,\theta_i}$ and 
\[
\psi_i(k,n) = \log \frac{1+ v_i G_{i,k}}{1+ v_i G_{i,k+n}} \le \log (1+ v_i G_{i,k}) = \psi_{i,k}^\star 
\]
for any $n \ge 1$. Since $\psi_{i,k}^\star \ge 0$ and $\{\Lc_{i,j}^{(2)}(\vartheta_i)\}_{j > k}$ are i.i.d.\ under  $\Pb_{k,i,\theta_i}$, we have
\begin{align*}
\Pb_{k,i,\theta_i}\brc{\frac{1}{n} \inf_{\vartheta_i\in \Gamma_{\delta,\theta_i}} \lambda_{\vartheta_i}(k,k+n) < I_{i,\theta_i}  - \varepsilon} & \le
\Pb_{k,i,\theta_i}\brc{\frac{1}{n} \inf_{\vartheta_i\in \Gamma_{\delta,\theta_i}} \sum_{j=k+1}^{k+n} \Lc_{i,j}^{(2)}(\vartheta_i)   < I_{i,\theta_i}  - \varepsilon  -  \frac{1}{n} \psi_{i,k}^\star}
\\
& \le \Pb_{k,i,\theta_i}\brc{\frac{1}{n} \inf_{\vartheta_i\in \Gamma_{\delta,\theta_i}} \sum_{j=k+1}^{k+n} \Lc_{i,j}^{(2)}(\vartheta_i)  <  I_{i,\theta_i}  - \varepsilon} 
\\
& = \Pb_{0,i,\theta_i}\brc{\frac{1}{n} \inf_{\vartheta_i\in \Gamma_{\delta,\theta_i}} \sum_{j=1}^{n} \Lc_{i,j}^{(2)}(\vartheta_i)  <  I_{i,\theta_i}  - \varepsilon} 
\end{align*}
and, consequently, conditions $\C_2^{(i)}$ are satisfied as long as
\begin{equation}\label{rcompLeftR}
\sum_{n=1}^\infty \, n^{r-1} \, \sup_{\theta_i \in \Theta_{i,c}}  
\Pb_{0,i,\theta_i}\brc{\frac{1}{n} \inf_{\vartheta_i\in \Gamma_{\delta,\theta_i}}  \sum_{j=1}^{n} \Lc_{i,j}^{(2)}(\vartheta_i) < I_{i,\theta_i}  - \varepsilon} <\infty .
\end{equation}
Typically condition \eqref{rcompLeftR} holds if the $(r+1)$th absolute moment of $\Lc_{i,1}^{(2)}(\vartheta_i)$ is finite, 
$\Eb_{0, i,\theta_i}|\Lc_{i,1}^{(2)}(\vartheta_i)|^{r+1}<\infty$. 

For example, let us consider the following Gaussian model: 
\begin{align*}
f_{i,\theta_i}(x) &= \frac{1}{\sqrt{2\pi \sigma_i^2}} \exp\set{\frac{(x- \theta_i)^2}{2\sigma_i^2}}, \quad p_{s,i}(x)=  \frac{1}{\sqrt{2\pi \sigma_i^2}} \exp\set{\frac{(x- \mu_{i,s})^2}{2\sigma_i^2}}, \quad s=1,2,
\end{align*}
where $\theta_i>0$,  $\mu_{i,1}>\mu_{i,2}=0$. Then 
\[
\Lc_{i,n}^{(s)}(\theta_i)= \frac{\theta_i-\mu_{i,s}}{\sigma_i^2} X_{i,n} - \frac{(\theta_i-\mu_{i,s})^2}{2 \sigma_i^2},  
\]
$I_{i,\theta_i}^{(s)} =  (\theta_i-\mu_{i,s})^2/2\sigma_i^2$, $I_{i,\theta_i}^{(2)}> I_{i,\theta_i}^{(1)}$ and
\[
\frac{1}{n} \sum_{j=1}^n \Lc_{i,j}^{(2)}(\vartheta_i,\theta_i) =  \frac{\vartheta_i \theta_i-\vartheta_i^2/2}{\sigma_i^2}    + \frac{\vartheta_i}{\sigma_i n} \sum_{j=1}^{n}\eta_{i,j},
\]
where $\eta_{i,j}\sim \Nc(0,1)$, $j=1,2,\dots$ are i.i.d.\ standard normal random variables. Since all moments of $\eta_{i,j}$ are finite, by the same argument as in the previous example, 
condition \eqref{rcompLeftR} holds for all $r\ge 1$, and hence,
the detection rule $T_{A_\alpha}^W$ is asymptotically optimal  as $\alpha\to0$, minimizing all positive moments of the detection delay.

\section{Discussion and remarks}\label{sec:Remarks}

1. Note that condition $\C_{1}$  holds whenever $\lambda_{\Bc,\theta}(k, k+n)/n$ converges almost surely to $I_{\Bc,\theta}$ under $\Pb_{k, \Bc,\theta}$,
\begin{equation}\label{sec:MaRe.1}
\frac{1}{n}\lambda_{\Bc,\theta}(k,k+n) \xra[n\to\infty]{\Pb_{k, \Bc,\theta}-\text{a.s.}} I_{\Bc,\theta}
\end{equation}
(cf.\ Lemma~A.1 in \cite{FellourisTartakovsky-IEEEIT2017}). However, the a.s.\ convergence is not sufficient for asymptotic optimality of the detection procedures with respect 
to moments of the detection delay.
In fact, the average detection delay may even be infinite under the a.s.\ convergence \eqref{sec:MaRe.1}. The left-tail condition $\C_2$ guarantees finiteness of first $r$ moments of the detection delay
and asymptotic optimality of the detection procedures in Theorem~\ref{Th:FOAODMS}, Theorem~\ref{Th:AoptDSR}, and Theorem~\ref{Th:FOasopt_pureBayes}. 
Note also that the uniform $r$-complete convergence conditions for  $n^{-1}\lambda_{\Bc,\theta}(k, k+n)$ and $n^{-1} \log \Lambda_{\pb,W}(k,k+n)$ to $I_{\Bc,\theta}$ under $\Pb_{k,\Bc,\theta}$, 
i.e., when for all $\varepsilon >0$, $\Bc\in \Pc$, and $\theta\in\Theta$
\begin{align*}
& \sum_{n=1}^\infty \, n^{r-1} \, \sup_{k \in \Zbb_+} \Pb_{k,\Bc,\theta}\brc{\abs{\frac{1}{n} \lambda_{\Bc,\theta}(k, k+n)-I_{\Bc,\theta}} > \varepsilon} <\infty , 
\\
&\sum_{n=1}^\infty \, n^{r-1} \, \sup_{k \in \Zbb_+} \Pb_{k, \Bc,\theta}\brc{\abs{\frac{1}{n}  \log \Lambda_{\pb,W}(k,k+n)-I_{\Bc,\theta}} > \varepsilon} <\infty ,
\end{align*}
are sufficient for asymptotic optimality results presented in Theorems~\ref{Th:FOAODMS}--\ref{Th:FOasopt_pureBayes}.  However, on the one hand these conditions are stronger than conditions 
$\C_1$ and $\C_2$, and on the other hand, verification of the $r$-complete convergence conditions is more difficult than checking conditions $\C_1$ and $\C_2$ for the local values of 
the LLR in the vicinity of the true parameter value, which is especially true for the weighted LLR $\log \Lambda_{\pb,W}(k,k+n)$. Still the  $r$-complete convergence conditions are intuitively 
appealing since they define the rate of convergence in the strong law of large numbers \eqref{sec:MaRe.1}.

2.  Even for independent streams the computational complexity and memory requirements of the procedures $T_A^W$ and $\wtT_A^W$ can be quite high. For this reason, in practice, 
it is reasonable to use window-limited versions of double-mixture detection procedures where the summation over potential change points $k$ is restricted to the sliding window of size 
$m=m_1-m_0$. The idea of using a window-limited generalized likelihood ratio procedure for stochastic dynamic systems described by linear state-space models belongs to 
Willsky and Jones \cite{willsky-ac76}, and a general (mostly minimax) single-stream quickest changepoint detection theory for window-limited CUSUM-type procedures based on the maximization over 
$k$ restricted to $n - m_1 \le k \le n$ was developed by Lai~\cite{LaiJRSS95,LaiIEEE98} who suggested a method of selection of $m_1$ (depending on the given false alarm rate) 
to make the detection procedures asymptotically optimal. The role of $m_1$ is to reduce the memory requirements and computational complexity of stopping rules. 
The values of $m_0$ bigger than $0$ can be used to protect against outliers, but $m_0=0$ looks reasonable in most cases. To be more specific, in the window-limited versions of 
$T_A^W$ and $\wtT_A^W$, defined in \eqref{DMS_def} and \eqref{DMSR_def}, the statistics $S_{\pb,W}^{\pi}(n)$ and $R_{\pb,W}(n)$ are replaced by the window-limited statistics
\begin{align*}
\widehat{S}_{\pb,W}^{\pi}(n) & = S_{\pb,W}^{\pi}(n) \quad \text{for}~n \le m_1; 
\\
\widehat{S}_{\pb,W}^{\pi}(n) & = \frac{1}{\Pb(\nu \ge n)} \sum_{k=n-(m_1+1)}^{n-1} \pi_k \Lambda_{\pb,W}(k,n) \quad \text{for}~n > m_1
\end{align*}
and
\begin{align*}
\widehat{R}_{\pb,W}(n)&= R_{\pb,W}(n)  \quad \text{for}~n \le m_1;
\\
\widehat{R}_{\pb,W}(n)& =  \sum_{k=n-(m_1+1)}^{n-1}  \Lambda_{\pb,W}(k,n) \quad \text{for}~n > m_1.
\end{align*}
Following guidelines of Lai~\cite{LaiIEEE98}, it can be shown that these window-limited versions also have first-order asymptotic optimality properties as long as the size of the 
window $m_1(A)$ approaches infinity as $A\to\infty$ with $m_1(A)/\log A \to \infty$ but $\log m_1(A)/\log A \to 0$. Since thresholds $A=A_\alpha$ in detection procedures should be selected 
in such a way that $\log A_\alpha \sim |\log \alpha|$ as $\alpha\to0$, it follows that the value of the window size $m_1(\alpha)$ should satisfy
\[
\lim_{\alpha\to0} \frac{m_1(\alpha)}{|\log \alpha|} =\infty, \quad  \lim_{\alpha\to0} \frac{\log m_1(\alpha)}{|\log \alpha|} =0.
\]

3. It is expected that first-order approximations to the moments of the detection delay are inaccurate in most cases, so higher-order approximations are in order. However, it is not feasible to 
obtain such approximations in the general non-i.i.d.\ case considered in Part 1 of the article. The author is currently working on the companion paper ``Asymptotically Optimal Quickest Change 
Detection in Multistream Data---Part 2: Higher-Order Approximations to Operating Characteristics in the i.i.d.\ Case,'' where we will derive higher-order approximations to the expected delay to detection 
and the probability of false alarm in the ``i.i.d.'' scenario, assuming that the observations in streams are independent and also independent across streams. The results of the renewal theory and 
nonlinear renewal theory will be used for this purpose. In the companion paper, we will also study the accuracy of asymptotic approximations and compare several detection schemes using MC simulations.



\appendix

\renewcommand{\theequation}{A.\arabic{equation}}
\setcounter{equation}{0}

\renewcommand{\thelemma}{A.\arabic{lemma}}
\setcounter{lemma}{0}

\section*{Appendix: An Auxiliary Lemma and Proofs}

The following lemma is extensively used for obtaining upper bounds for the moments of the detection delay, which are needed for proving asymptotic optimality properties of 
the introduced detection procedures. In this lemma, $\Pb$ is a generic probability measure and $\Eb$ is a corresponding expectation.

\begin{lemma}\label{LemmaA1}
Let $\tau$ ($\tau=0,1,\dots$) be a non-negative integer-valued random variable and let $N$ ($N \ge 1$) be an integer number. Then, for any $r \ge 1$,
\begin{equation}\label{upperineqtau}
 \Eb[\tau]^r \le N^{r} + r 2^{r-1} \sum_{n=N}^{\infty}  n^{r-1}   \Pb\brc{\tau >  n}.
 \end{equation}
\end{lemma}

\proof
\begin{align*}
& \Eb\brcs{\tau}^r   = \int_0^\infty r t^{r-1} \Pb\brc{\tau > t} \, \mrm{d} t   \nonumber
\\
& \le N^r + \sum_{n=0}^{\infty} \int_{N+n}^{N+n+1} r t^{r-1}  \Pb \brc{\tau > t} \, \mrm{d} t  \nonumber
\\
& \le  N^r + \sum_{n=0}^{\infty} \int_{N+n}^{N+n+1} r t^{r-1}  \Pb\brc{\tau > N+n} \, \mrm{d} t  \nonumber
\\
& = N^r + \sum_{n=0}^{\infty} [(N+n+1)^r- (N+n)^r ] \Pb\brc{\tau > N+n} \nonumber
\\
& = N^r + \sum_{n=N}^{\infty} [(n+1)^r-n^r]  \Pb\brc{\tau > n}  \nonumber
\\
 & \le  N^{r} +\sum_{n=N}^{\infty}   r (n+1)^{r-1}   \Pb\brc{\tau >  n} \nonumber
 \\
 & \le N^{r} + r 2^{r-1} \sum_{n=N}^{\infty}  n^{r-1}   \Pb\brc{\tau >  n}.
 \end{align*}
\endproof

\begin{proof}[Proof of Proposition~\ref{Lem:AOCDMS}]
To prove asymptotic approximations \eqref{MomentskDMS} and \eqref{MomentsDMS} note first that by \eqref{PFADMSineq} the detection procedure $T_A^W$ belongs to class $\mbb{C}(1/(A+1))$,
so replacing $\alpha$ by $1/(A+1)$ in the asymptotic lower bounds \eqref{LBkinclass} and \eqref{LBinclass}, we obtain that under the right-tail condition $\C_1$ 
the following asymptotic lower bounds hold for all $r>0$, $\Bc\in \Pc$, and $\theta\in\Theta$:
\begin{align}
\label{LBkTAW}
 \liminf_{A\to\infty} \frac{\Rc_{k, \Bc,\theta}^r(T_A^W)}{(\log A)^r} & \ge \frac{1}{(I_{\Bc,\theta} +\mu)^r} , \quad k \in \Zbb_+,
\\
 \liminf_{A\to\infty} \frac{\Rca_{\Bc,\theta}^r(T_A^W)}{(\log A)^r} & \ge \frac{1}{(I_{\Bc,\theta} +\mu)^r} . 
\label{LBTAW}
\end{align}
Therefore, to prove the assertions of the proposition it suffices to show that, under the left-tail condition $\C_2$, for all $0 < m \le r$, $\Bc\in\Pc$, and $\theta\in\Theta$
\begin{align}
\label{UBkTAW}
 \limsup_{A\to\infty} \frac{\Rc_{k, \Bc,\theta}^m(T_A^W)}{(\log A)^m} & \le \frac{1}{(I_{\Bc,\theta} +\mu)^m} ,  \quad k \in \Zbb_+,
\\
 \limsup_{A\to\infty} \frac{\Rca_{\Bc,\theta}^m(T_A^W)}{(\log A)^m} & \le \frac{1}{(I_{\Bc,\theta} +\mu)^m} .
\label{UBTAW}
\end{align}

The proof of part (i). Let $\pi^A=\{\pi_k^A\}$, $\pi_k^A=\pi_k^{\alpha}$ for $\alpha=\alpha_A = 1/(1+A)$, and define
\begin{equation*}
N_A=N_{A}(\varepsilon,\Bc,\theta)=1+\left \lfloor \frac{\log (A/\pi_k^A)}{I_{\Bc,\theta}+\mu-\varepsilon} \right \rfloor .
\end{equation*}

Obviously,  for any $n \ge 1$,
\begin{align*}
\log S_{\pb,W}^\pi(k+n) & \ge  \log \Lambda_{\pb,W}(k,k+n) + \log \pi_k^A -\log \Pi_{k-1+n}^A 
\\
& \ge \inf_{\vartheta \in \Gamma_{\delta,\theta}} \lambda_{\Bc,\vartheta}(k, k+n) +\log W(\Gamma_{\delta,\theta}) +\log p_\Bc  +\log \pi_k^A -\log \Pi_{k-1+n}^A , 
\end{align*}
where $\Gamma_{\delta,\theta}=\{\vartheta\in\Theta\,:\,\vert\vartheta-\theta\vert<\delta\}$, so that for any $\Bc\in\Pc$, $\theta\in\Theta$, $k\in\Zbb_+$
\begin{align*}
&\Pb_{k, \Bc,\theta}\brc{T_A^W-k >n} \le \Pb_{k, \Bc,\theta}\set{\frac{1}{n} \log S_{\pb,W}^\pi(k+n) < \frac{1}{n} \log A }\nonumber
\\
& \le \Pb_{k, \Bc,\theta}\set{\frac{1}{n} \inf_{\vartheta \in \Gamma_{\delta,\theta}} \lambda_{\Bc,\vartheta}(k, k+n)  <  \frac{1}{n} \log \brc{\frac{A}{\pi_k^A}} +\frac{1}{n}  \brcs{|\log\Pi_{k-1+n}^A| + 
\log W(\Gamma_{\delta,\theta}) +\log p_\Bc}} .
\end{align*}
It is easy to see that for $n\ge N_{A}$ the last probability does not exceed the probability
\[
\Pb_{k, \Bc,\theta}\set{\frac{1}{n} \inf_{\vartheta \in \Gamma_{\delta,\theta}} \lambda_{\Bc,\vartheta}(k, k+n) <  I_{\Bc,\theta} +\mu - \varepsilon +\frac{1}{n}  \brc{|\log\Pi_{k-1+n}^A| + 
\log W(\Gamma_{\delta,\theta}) +\log p_\Bc}}.
\]
Since, by condition $\mb{CP 1}$, $N_{A}^{-1} |\log\Pi_{k-1+N_A}^A| \to \mu$ as $A\to \infty$, for a sufficiently large value of $A$ there exists a small 
$\kappa=\kappa_A$ ($\kappa_A\to 0$ as $A\to\infty$) such that 
\begin{equation}\label{logPikappa}
\left |\mu - \frac{|\log\Pi_{k-1+N_A}^A|}{N_{A}} \right | < \kappa.
\end{equation} 
Hence, for $\varepsilon_1 >0$ and all sufficiently large $A$ and $n$, we have
\begin{align}
 \Pb_{k, \Bc,\theta}\brc{T_A^W-k >n} & \le \Pb_{k,\Bc}\set{\frac{1}{n}  \inf_{\vartheta \in \Gamma_{\delta,\theta}} \lambda_{\Bc,\vartheta}(k, k+n)   < I_{\Bc,\theta}  - \varepsilon- \kappa - 
\frac{1}{n}\brcs{\log p_\Bc + \log W(\Gamma_{\delta,\theta})}}
\nonumber
\\
&  \le \Pb_{k,\Bc,\theta}\set{\frac{1}{n}  \inf_{\vartheta \in \Gamma_{\delta,\theta}} \lambda_{\Bc,\vartheta}(k, k+n)  < I_{\Bc,\theta}  - \varepsilon_1}. \label{ProbkTAW}
\end{align}
By Lemma~\ref{LemmaA1}, for any $k\in\Zbb_+$,  $\Bc\in\Pc$, and $\theta\in\Theta$ we have the following inequality
\begin{align}\label{EkTAWineq}
 \Eb_{k,\Bc,\theta}\brcs{(T_A^W-k)^+}^r    \le N_A^{r} + r 2^{r-1} \sum_{n=N_A}^{\infty} n^{r-1}   \Pb_{k,\Bc,\theta}\brc{T_A^W-k >  n},
 \end{align}
 which along with \eqref{ProbkTAW} yields
\begin{align}
\Eb_{k,\Bc,\theta}\brcs{\brc{T_A^W-k}^+}^r
 \le   \brc{1+ \left \lfloor\frac{ \log (A/\pi_k^A)}{I_{\Bc,\theta}+\mu-\varepsilon}\right \rfloor}^r + r 2^{r-1} \, \Upsilon_r(\varepsilon_1,\Bc, \theta). 
 \label{EkTAWkupper}
 \end{align}

Now, note that
\[
\PFA(T_A^W) \ge \sum_{i=k}^\infty \pi_i^A \Pb_\infty(T_A^W \le i) \ge \Pb_\infty(T_A^W \le k) \sum_{i=k}^\infty \pi_i^A = \Pb_\infty(T_A^W \le k) \Pi_{k-1}^A ,
\]
and hence,
\begin{equation}\label{PFATAWge}
 \Pb_\infty(T_A^W > k) \ge 1- \PFA(T_A^W)/\Pi_{k-1}^A \ge 1- [(A+1) \Pi_{k-1}^A]^{-1} , \quad k \in \Zbb_+.
\end{equation}
Recall that we set $\Pi_k^A=\Pi^\alpha_k$ with $\alpha=\alpha_A =1/(1+A)$. It follows from \eqref{EkTAWkupper} and \eqref{PFATAWge} that 
 \begin{align}
  \Rc_{k, \Bc,\theta}^r(T_A^W)& =\frac{\Eb_{k, \Bc,\theta}\brcs{\brc{T_A^W-k}^+}^r}{\Pb_\infty(T_A^W > k)}  \nonumber
  \\
  &\le   \frac{\brc{1+\left \lfloor \frac{\log (A/\pi_k^A)}{I_{\Bc,\theta}+\mu-\varepsilon} \right \rfloor}^r + r 2^{r-1} \, \Upsilon_r(\varepsilon_1,\Bc, \theta)}{1- 1/(A \Pi^A_{k-1})}. 
  \label{RckupperTAW}
 \end{align}
Since by condition $\C_2$, $\Upsilon_{r}(\varepsilon_1,\Bc, \theta) < \infty$ for all $\Bc\in\Pc$, $\theta\in\Theta$, and $\varepsilon_1 >0$  and, by condition $\mb{CP} {\mb 3}$,
$(A\Pi_{k-1}^A)^{-1} \to 0$, $|\log \pi_k^A|/\log A \to 0$ as $A\to\infty$, inequality \eqref{RckupperTAW} implies the asymptotic inequality
\[
\Rca_{k,\Bc,\theta}^r(T_A^W) \le \brc{\frac{\log A}{I_{\Bc,\theta}+\mu-\varepsilon}}^r (1+o(1)), \quad A \to \infty.
\]
Since $\varepsilon$ can be arbitrarily small, this implies the asymptotic upper bound  \eqref{UBkTAW} 
(for all $0<m \le r$, $\Bc\in\Pc$, and $\theta\in\Theta$).  This upper bound and the lower bound \eqref{LBkTAW} prove the asymptotic relation \eqref{MomentskDMS}.  The proof of (i) is complete.

The proof of part (ii). Using the inequalities \eqref{RckupperTAW} and $1-\PFA(T_A^W) \ge A/(1+A)$, we obtain that for any $0<\varepsilon < I_{\Bc,\theta}+\mu$
\begin{equation}\label{UpperRcaTAW}
\begin{split}
&\Rca_{\Bc,\theta}^r(T_A^W) = \frac{\sum_{k=0}^\infty \pi_k^A  \Eb_{k,\Bc,\theta}\brcs{(T_A^W-k)^+}^r}{1-\PFA(T_A^W)} 
\\
&\le\frac{{\displaystyle\sum_{k=0}^\infty}  \pi_k^A \brc{1+\left \lfloor\frac{\log (A/\pi_k^A)}{I_{\Bc,\theta}+\mu-\varepsilon} \right \rfloor}^r + 
r 2^{r-1} \, \Upsilon_r(\varepsilon_1,\Bc,\theta)}{A/(1+A)} .
\end{split}
\end{equation}
By condition $\C_2$, $\Upsilon_r(\varepsilon_1,\Bc,\theta) < \infty$ for any $\varepsilon_1 >0$,   $\Bc\in\Pc$, and $\theta\in\Theta$
and, by condition $\mb{CP} {\mb 2}$, $\sum_{k=0}^\infty \pi_k^A |\log\pi_k^A|^r =o(|\log A|^r)$ as $A\to\infty$, which implies that for all $\Bc\in\Pc$ and $\theta\in\Theta$
\[
\Rca_{\Bc,\theta}^r(T_A^W) \le \brc{\frac{\log A}{I_{\Bc,\theta}+\mu - \varepsilon}}^r (1+o(1)), \quad A \to \infty.
\]
Since $\varepsilon$ can be arbitrarily small, the asymptotic upper bound  \eqref{UBTAW} follows and the proof of the asymptotic approximation \eqref{MomentsDMS} is complete.  
\end{proof}

\begin{proof}[Proof of Proposition~\ref{Lem:AOCDMSR}]
As before, $\pi_k^A=\pi_k^{\alpha_A}$, so $\bar\nu_A=\bar\nu_{\alpha_A}$  and $\omega_A=\omega_{\alpha_A}$, where  $|\log \alpha_A|\sim \log A$. 

For $\varepsilon\in(0,1)$, let 
\begin{equation*}
M_A=M_{A}(\varepsilon,\Bc,\theta) = (1-\varepsilon) \frac{\log A}{I_{\Bc,\theta}} . 
\end{equation*}
Recall that
\[
\Pb_{k, \Bc,\theta}(\wtT_A^W>k)=\Pb_\infty (\wtT_A^W>k) \ge 1- \frac{k+ \omega_A} {A} , \quad k \in \Zbb_+
\]
(see \eqref{PFADoobineqDM}), so using Chebyshev's inequality, we obtain
\begin{align} \label{LBSRA}
\Rc_{k,\Bc,\theta}^r(\wtT_A^W)
& \ge M_{A}^r \Pb_{k, \Bc,\theta}(\wtT_A^W  -k > M_{A})  \nonumber
\\
&\ge M_{A}^r\brcs{\Pb_{k,\Bc,\theta}(\wtT_A^W>k) - \Pb_{k, \Bc,\theta}(k <  \wtT_A^W < k+M_{A})}  \nonumber
\\
&\ge  M_{A}^r\brcs{1- \frac{\omega_A  +k } {A} - \Pb_{k, \Bc,\theta}(k <  \wtT_A^W < k+M_{A})}.
\end{align}

Analogously to \eqref{PkTupper},
\begin{equation}\label{Pktauupper}
 \Pb_{k, \Bc,\theta}\brc{k <  T < k+ M_{A}} \le  U_{M_A,k}( T)  + \beta_{M_A,k}(\varepsilon, \Bc,\theta).
\end{equation}
 Since 
\[
\begin{aligned}
\Pb_\infty\brc{0 < \wtT_A^W - k <M_{A}} & \le \Pb_\infty\brc{\wtT_A^W  < k+ M_{A}} 
\\
&\le (k+\omega_A + M_{A})/A,
\end{aligned}
\]
we have
\begin{equation}\label{UpperU}
 U_{M_A,k}(\wtT_A^W) \le \frac{k+ \omega_A+(1-\varepsilon) I_{\Bc,\theta}^{-1} \log A}{A^{\varepsilon^2}}.
\end{equation}
By condition \eqref{Prior4},
\begin{equation}\label{PriorA}
\lim_{A\to \infty} \frac{{\log (\omega_A+\bar\nu_A)}}{\log A} = 0,
\end{equation} 
which implies that $\omega_A=o(A^\gamma)$ as $A\to\infty$ for any $\gamma>0$.
Therefore,  $U_{M_A,k}(\wtT_A^W)\to 0$ as $A\to\infty$ for any fixed $k$. Also, $\beta_{M_A,k}(\varepsilon,\Bc,\theta)\to 0$ by condition $\C_1$, so that 
$\Pb_{k, \Bc,\theta}\brc{0 < \wtT_A^W -k < M_{A}}\to0$ for any fixed $k$. It follows from \eqref{LBSRA}  that for an arbitrary $\varepsilon \in (0,1)$ as $A\to\infty$
\begin{equation*}
\Rc_{k,\Bc,\theta}^r(\wtT_A^W) \ge \brc{\frac{(1-\varepsilon) \log A}{I_{\Bc,\theta}}}^r (1+o(1)),
\end{equation*}
which yields the asymptotic lower bound (for any fixed $k\in\Zbb_+$, $\Bc\in\Pc$, and $\theta\in\Theta$)
\begin{equation}\label{LBSRAasympt}
\liminf_{A\to\infty}\frac{\Rc_{k,\Bc,\theta}^r(\wtT_A^W)}{(\log A)^r} \ge \frac{1}{I_{\Bc,\theta}^r} .
\end{equation}

To prove \eqref{MomentskDMSR} it suffices to show that this bound is attained by $\wtT_A^W$, i.e., 
\begin{equation}\label{UBkTDSR}
 \limsup_{A\to\infty} \frac{\Rc_{k, \Bc,\theta}^r(\wtT_A^W)}{(\log A)^r}  \le \frac{1}{I_{\Bc,\theta}^r} .
\end{equation}

Define 
\begin{equation*}
\widetilde{M}_{A}= \widetilde{M}_{A}(\varepsilon,\Bc,\theta)=1+\left\lfloor \frac{\log A}{I_{\Bc,\theta}-\varepsilon} \right\rfloor.
\end{equation*}
By Lemma~\ref{LemmaA1}, for any $k\in\Zbb_+$, $\Bc\in\Pc$, and $\theta\in\Theta$, 
\begin{align}\label{Ektildetauineq}
& \Eb_{k, \Bc,\theta}\brcs{(\wtT_A^W-k)^+}^r   \le 
 \widetilde{M}_{A}^{r} + r 2^{r-1} \sum_{n= \widetilde{M}_{A}}^{\infty}  n^{r-1}   \Pb_{k, \Bc,\theta}\brc{\wtT_A^W >  n},
 \end{align}
 and since for any $n \ge 1$,
\[
\log R_{\pb,W}^\pi(k+n) \ge  \log \Lambda_{\pb,W}(k,k+n)  \ge  \inf_{\vartheta \in \Gamma_{\delta,\theta}} \lambda_{\Bc,\vartheta}(k, k+n) + \log W(\Gamma_{\delta,\theta})+\log p_\Bc ,
\] 
in just the same way as in the proof of Proposition~\ref{Lem:AOCDMS} (setting $\pi_k^A=1$) we obtain that for all $n \ge \widetilde{M}_A$
\[
\Pb_{k, \Bc,\theta}\brc{\wtT_A^W >  n} \le 
\Pb_{k, \Bc,\theta}\set{\frac{1}{n} \inf_{\vartheta \in \Gamma_{\delta,\theta}} \lambda_{\Bc,\vartheta}(k, k+n) <  I_{\Bc,\theta} - \varepsilon +\frac{1}{n}  \brcs{\log W(\Gamma_{\delta,\theta}) +\log p_\Bc}}.
\]
Hence, for all sufficiently large  $n$ and $\varepsilon_1 >0$,
\begin{align} \label{ProbkTWSR}
\Pb_{k, \Bc,\theta}\brc{\wtT_A^W-k >n}  
&  \le \Pb_{k,\Bc,\theta}\brc{\frac{1}{n}\inf_{\vartheta \in \Gamma_{\delta,\theta}} \lambda_{\Bc,\vartheta}(k, k+n)  < I_{\Bc,\theta}  - \varepsilon_1}. 
\end{align}
Using \eqref{Ektildetauineq} and \eqref{ProbkTWSR}, we obtain
\begin{equation}\label{EkinequppertildeTAW}
\Eb_{k, \Bc,\theta}\brcs{\brc{\wtT_A^W-k}^+}^r \le \brc{1+\left\lfloor \frac{\log A}{I_{\Bc,\theta}-\varepsilon} \right\rfloor}^r + r 2^{r-1} \, \Upsilon_{r}(\varepsilon_1,\Bc, \theta),
\end{equation}
which along with the inequality $\Pb_\infty(\wtT_A^W > k) > 1- (\omega_A +k)/A$ (see \eqref{PFADoobineqDM}) implies the inequality
\begin{align}
  \Rc_{k, \Bc,\theta}^r(\wtT_A^W)& =\frac{\Eb_{k, \Bc,\theta}\brcs{\brc{\wtT_A^W-k}^+}^r}{\Pb_\infty(\wtT_A^W > k)}  \nonumber
  \\
  &\le   \frac{\brc{1+\left\lfloor \frac{\log A}{I_{\Bc,\theta}-\varepsilon} \right\rfloor}^r + r 2^{r-1} \, \Upsilon_{r}(\varepsilon_1,\Bc, \theta)}{1- (\omega_A +k)/A}. 
  \label{RckupperDSR}
 \end{align}
Since due to \eqref{PriorA} $\omega_A/A\to 0$ and, by condition $\C_2$, $\Upsilon_{r}(\varepsilon_1,\Bc, \theta) < \infty$ for all $\varepsilon_1>0$, $\Bc\in\Pc$, $\theta\in\Theta$, inequality
\eqref{RckupperDSR} implies the asymptotic inequality
\[
 \Rc_{k, \Bc,\theta}^r(\wtT_A^W)\le \brc{\frac{\log A}{I_{\Bc,\theta} - \varepsilon}}^r (1+o(1)), \quad A \to \infty.
\]
Since $\varepsilon$ can be arbitrarily small the asymptotic upper bound \eqref{UBkTDSR} follows
and  the proof of the asymptotic approximation  \eqref{MomentskDMSR} is complete.

In order to prove  \eqref{MomentsDMSR} note first that, using \eqref{LBSRA}, yields the lower bound
\begin{align} \label{LBMA}
 \Rca_{\Bc,\theta}^r(\wtT_A) & \ge M_{A}^r\brcs{1-\frac{\bar\nu_A+\omega_A}{A} - \Pb^{\pi}_{\Bc,\theta} \brc{0 < \wtT_A-\nu < M_{A}}} .
\end{align} 
Let $K_A$ be an integer number that approaches infinity as $A\to\infty$ with rate $O(A^\gamma)$, $\gamma>0$. Now, using \eqref{Pktauupper} and \eqref{UpperU}, we obtain 
\begin{align}\label{ProbineqSR}
&\Pb^{\pi}_{\Bc,\theta}(0< \wtT_A^W -\nu < M_{A}) 
= \sum_{k=0}^\infty  \pi_k^A  \Pb_{k,\Bc,\theta}\brc{0 < \wtT_A^W -k < M_{A}} \nonumber
 \\
 &\le \Pb(\nu > K_{A}) +   \sum_{k=0}^\infty \pi_k^A U_{M_A,k}(\wtT_A^W) + \sum_{k=0}^{K_{A}}  \pi_k^A \beta_{M_A, k}(\varepsilon,\Bc,\theta) \nonumber
 \\ 
 & \le  \Pb(\nu > K_{A}) + \frac{\bar\nu_A+\omega_A+ (1-\varepsilon) I_{\Bc,\theta}^{-1} \log A}{A^{\varepsilon^2}} 
  +  \sum_{k=0}^{K_{A}}  \pi_k^A  \beta_{M_A,k}(\varepsilon,\Bc,\theta) .
 \end{align}
Note that due to \eqref{PriorA} $(\omega_A+\bar\nu_A)/A^\gamma \to 0$ as $A\to\infty$ for any $\gamma>0$. As a result, 
the first two  terms in \eqref{ProbineqSR} go to zero as $A\to\infty$ (by Markov's inequality $\Pb(\nu>K_A) \le \bar\nu_A/K_A = \bar\nu_A/O(A^\gamma) \to 0$) 
and the last term also goes to zero  
by condition $\A_1$ and Lebesgue's dominated convergence theorem. Thus, for all $0<\varepsilon <1$,
$\Pb^\pi_\Bc(0< \wtT_A^W -\nu < M_{A})$ approaches $0$ as $A\to\infty$.
Using inequality \eqref{LBMA}, we obtain that for any $0<\varepsilon <1$ as $A\to\infty$
\[
\Rca_{\Bc,\theta}^r(\wtT_A^W) \ge (1-\varepsilon)^r \brc{\frac{\log A}{I_{\Bc,\theta}}}^r (1+o(1)),
\]
which yields the asymptotic lower bound (for any $r>0$, $\Bc\in\Pc$, and $\theta\in\Theta$)
\begin{equation}\label{LBTAWRca}
\liminf_{A\to\infty}\frac{\Rca_{\pi,\Bc,\theta}^r(\wtT_A^W)}{(\log A)^r}  \ge \frac{1}{I_{\Bc,\theta}^r}.
\end{equation}

To obtain the upper bound it suffices to use inequality \eqref{EkinequppertildeTAW}, which along with the fact that $\PFA(\wtT_A^W) \le (\bar\nu_A +\omega_A)/A$ yields 
(for every $0<\varepsilon < I_{\Bc,\theta}$)
\begin{align*}
&\Rca_{\Bc,\theta}^r(\wtT_A^W) = \frac{\sum_{k=0}^\infty \pi_k^A  \Eb_{k,\Bc,\theta}[(\wtT_A^W-k)^+]^r}{1-\PFA(\wtT_A^W )} \nonumber
\\
&\le\frac{\brc{1+\frac{\log A}{I_{\Bc,\theta}-\varepsilon}}^r + r 2^{r-1}\, \Upsilon_{r}(\varepsilon_1,\Bc,\theta)}{1-(\omega_A+\bar\nu_A)/A} .
\end{align*}
Since $(\omega_A+\bar\nu_A)/A\to 0$ and, by condition $\C_2$, $\Upsilon_{r}(\varepsilon_1,\Bc,\theta) < \infty$ for any $\varepsilon >0$, $\Bc\in\Pc$, and $\theta\in\Theta$
we obtain that, for every $0<\varepsilon < I_{\Bc,\theta}$ as $A \to \infty$,
\[
\Rca_{\Bc,\theta}^r(\wtT_A^W) \le \brc{\frac{\log A}{I_{\Bc,\theta}-\varepsilon}}^r (1+o(1)) ,
\]
which implies 
\begin{equation}\label{UBTDMSR}
 \limsup_{A\to\infty} \frac{\Rca_{\Bc,\theta}^r(\wtT_A^W)}{(\log A)^r}  \le \frac{1}{I_{\Bc,\theta}^r} 
\end{equation}
since $\varepsilon$ can be arbitrarily small.

Applying the bounds  \eqref{LBTAWRca} and \eqref{UBTDMSR} together completes the proof of  \eqref{MomentsDMSR}.
\end{proof}

%

\begin{thebibliography}{19}

\bibitem{Bakutetal-book63}
\begin{bbook}[author]
\bauthor{\bsnm{Bakut},~\bfnm{P.~A.}\binits{P.~A.}},
  \bauthor{\bsnm{Bolshakov},~\bfnm{I.~A.}\binits{I.~A.}},
  \bauthor{\bsnm{Gerasimov},~\bfnm{B.~M.}\binits{B.~M.}},
  \bauthor{\bsnm{Kuriksha},~\bfnm{A.~A.}\binits{A.~A.}},
  \bauthor{\bsnm{Repin},~\bfnm{V.~G.}\binits{V.~G.}},
  \bauthor{\bsnm{Tartakovsky},~\bfnm{G.~P.}\binits{G.~P.}} \AND
  \bauthor{\bsnm{Shirokov},~\bfnm{V.~V.}\binits{V.~V.}}
(\byear{1963}).
\btitle{Statistical Radar Theory}
\bvolume{1 (G. P. Tartakovsky, Editor)}.
\bpublisher{Sovetskoe Radio}, \baddress{Moscow, USSR}.
\bnote{In Russian}.
\end{bbook}
\endbibitem

\bibitem{Chan-AS2017}
\begin{barticle}[author]
\bauthor{\bsnm{Chan},~\bfnm{Hock~Peng}\binits{H.~P.}}
(\byear{2017}).
\btitle{Optimal sequential detection in multi-stream data}.
\bjournal{Annals of Statistics}
\bvolume{45}
\bpages{2736--2763}.
\end{barticle}
\endbibitem

\bibitem{felsokIEEEIT2016}
\begin{barticle}[author]
\bauthor{\bsnm{Fellouris},~\bfnm{Georgios}\binits{G.}} \AND
  \bauthor{\bsnm{Sokolov},~\bfnm{Gregory}\binits{G.}}
(\byear{2016}).
\btitle{Second-order asymptotic optimality in multichannel sequential
  detection}.
\bjournal{IEEE Transactions on Information Theory}
\bvolume{62}
\bpages{3662--3675}.
\bdoi{10.1109/TIT.2016.2549042}
\end{barticle}
\endbibitem

\bibitem{FellourisTartakovsky-IEEEIT2017}
\begin{barticle}[author]
\bauthor{\bsnm{Fellouris},~\bfnm{Georgios}\binits{G.}} \AND
  \bauthor{\bsnm{Tartakovsky},~\bfnm{Alexander~G.}\binits{A.~G.}}
(\byear{2017}).
\btitle{Multichannel sequential detection---{P}art {I}: Non-i.i.d. data}.
\bjournal{IEEE Transactions on Information Theory}
\bvolume{63}
\bpages{4551--4571}.
\bdoi{10.1109/TIT.2017.2689785}
\end{barticle}
\endbibitem

\bibitem{LaiJRSS95}
\begin{barticle}[author]
\bauthor{\bsnm{Lai},~\bfnm{Tze~Leung}\binits{T.~L.}}
(\byear{1995}).
\btitle{Sequential changepoint detection in quality control and dynamical
  systems (with discussion)}.
\bjournal{Journal of the Royal Statistical Society - Series B Methodology}
\bvolume{57}
\bpages{613--658}.
\end{barticle}
\endbibitem

\bibitem{LaiIEEE98}
\begin{barticle}[author]
\bauthor{\bsnm{Lai},~\bfnm{Tze~Leung}\binits{T.~L.}}
(\byear{1998}).
\btitle{Information bounds and quick detection of parameter changes in
  stochastic systems}.
\bjournal{IEEE Transactions on Information Theory}
\bvolume{44}
\bpages{2917--2929}.
\end{barticle}
\endbibitem

\bibitem{Mei-B2010}
\begin{barticle}[author]
\bauthor{\bsnm{Mei},~\bfnm{Yajun}\binits{Y.}}
(\byear{2010}).
\btitle{Efficient scalable schemes for monitoring a large number of data
  streams}.
\bjournal{Biometrika}
\bvolume{97}
\bpages{419--433}.
\end{barticle}
\endbibitem

\bibitem{PollakTartakovsky-SS09}
\begin{barticle}[author]
\bauthor{\bsnm{Pollak},~\bfnm{M.}\binits{M.}} \AND
  \bauthor{\bsnm{Tartakovsky},~\bfnm{A.~G.}\binits{A.~G.}}
(\byear{2009}).
\btitle{Optimality properties of the {Shiryaev--Roberts} procedure}.
\bjournal{Statistica Sinica}
\bvolume{19}
\bpages{1729--1739}.
\end{barticle}
\endbibitem

\bibitem{PolunTartakovskyAS09}
\begin{barticle}[author]
\bauthor{\bsnm{Polunchenko},~\bfnm{A.~S.}\binits{A.~S.}} \AND
  \bauthor{\bsnm{Tartakovsky},~\bfnm{A.~G.}\binits{A.~G.}}
(\byear{2010}).
\btitle{On optimality of the {Shiryaev--Roberts} procedure for detecting a
  change in distribution}.
\bjournal{Annals of Statistics}
\bvolume{38}
\bpages{3445--3457}.
\end{barticle}
\endbibitem

\bibitem{TartakovskyIEEECDC05}
\begin{binproceedings}[author]
\bauthor{\bsnm{Tartakovsky},~\bfnm{A.~G.}\binits{A.~G.}}
(\byear{2005}).
\btitle{Asymptotic performance of a multichart {CUSUM} test under false alarm
  probability constraint}.
In \bbooktitle{Proceedings of the 44th IEEE Conference Decision and Control and
  European Control Conference (CDC-ECC'05), Seville, SP}
\bpages{320--325}.
\borganization{IEEE}.
\bpublisher{Omnipress CD-ROM}.
\end{binproceedings}
\endbibitem

\bibitem{TartakovskyIEEEIT2017}
\begin{barticle}[author]
\bauthor{\bsnm{Tartakovsky},~\bfnm{A.~G.}\binits{A.~G.}}
(\byear{2017}).
\btitle{On asymptotic optimality in sequential changepoint detection: Non-iid
  case}.
\bjournal{IEEE Transactions on Information Theory}
\bvolume{63}
\bpages{3433--3450}.
\bdoi{10.1109/TIT.2017.2683496}
\end{barticle}
\endbibitem

\bibitem{TartakovskyIEEEIT2018}
\begin{barticle}[author]
\bauthor{\bsnm{Tartakovsky},~\bfnm{Alexander~G.}\binits{A.~G.}}
(\byear{2018, under review}).
\btitle{Asymptotic optimality of mixture rules for detecting changes in general
  stochastic models}.
\bjournal{IEEE Transactions on Information Theory}.
\end{barticle}
\endbibitem

\bibitem{Tartakovsky&Brown-IEEEAES08}
\begin{barticle}[author]
\bauthor{\bsnm{Tartakovsky},~\bfnm{A.~G.}\binits{A.~G.}} \AND
  \bauthor{\bsnm{Brown},~\bfnm{J.}\binits{J.}}
(\byear{2008}).
\btitle{Adaptive spatial-temporal filtering methods for clutter removal and
  target tracking}.
\bjournal{IEEE Transactions on Aerospace and Electronic Systems}
\bvolume{44}
\bpages{1522--1537}.
\end{barticle}
\endbibitem

\bibitem{TNB_book2014}
\begin{bbook}[author]
\bauthor{\bsnm{Tartakovsky},~\bfnm{A.~G.}\binits{A.~G.}},
  \bauthor{\bsnm{Nikiforov},~\bfnm{I.~V.}\binits{I.~V.}} \AND
  \bauthor{\bsnm{Basseville},~\bfnm{M.}\binits{M.}}
(\byear{2014}).
\btitle{Sequential Analysis: Hypothesis Testing and Changepoint Detection}.
\bseries{Monographs on Statistics and Applied Probability}.
\bpublisher{Chapman \& Hall/CRC Press}, \baddress{Boca Raton, London, New
  York}.
\end{bbook}
\endbibitem

\bibitem{tartakovsky-tpa11}
\begin{barticle}[author]
\bauthor{\bsnm{Tartakovsky},~\bfnm{Alexander~G.}\binits{A.~G.}},
  \bauthor{\bsnm{Pollak},~\bfnm{Moshe}\binits{M.}} \AND
  \bauthor{\bsnm{Polunchenko},~\bfnm{Aleksey~S.}\binits{A.~S.}}
(\byear{2012}).
\btitle{Third-order asymptotic optimality of the generalized
  {Shiryaev--Roberts} changepoint detection procedures}.
\bjournal{Theory of Probability and its Applications}
\bvolume{56}
\bpages{457-484}.
\bdoi{10.1137/S0040585X97985534}
\end{barticle}
\endbibitem

\bibitem{Tartakovskyetal-SM06}
\begin{barticle}[author]
\bauthor{\bsnm{Tartakovsky},~\bfnm{Alexander~G.}\binits{A.~G.}},
  \bauthor{\bsnm{Rozovskii},~\bfnm{Boris~L.}\binits{B.~L.}},
  \bauthor{\bsnm{Bla\'{z}ek},~\bfnm{Rudolf~B.}\binits{R.~B.}} \AND
  \bauthor{\bsnm{Kim},~\bfnm{Hongjoong}\binits{H.}}
(\byear{2006}).
\btitle{Detection of intrusions in information systems by sequential
  change-point methods}.
\bjournal{Statistical Methodology}
\bvolume{3}
\bpages{252--293}.
\end{barticle}
\endbibitem

\bibitem{TartakovskyVeerTVP05}
\begin{barticle}[author]
\bauthor{\bsnm{Tartakovsky},~\bfnm{Alexander~G.}\binits{A.~G.}} \AND
  \bauthor{\bsnm{Veeravalli},~\bfnm{Venugopal~V.}\binits{V.~V.}}
(\byear{2005}).
\btitle{General asymptotic {Bayesian} theory of quickest change detection}.
\bjournal{Theory of Probability and its Applications}
\bvolume{49}
\bpages{458--497}.
\end{barticle}
\endbibitem

\bibitem{willsky-ac76}
\begin{barticle}[author]
\bauthor{\bsnm{Willsky},~\bfnm{Alan~S.}\binits{A.~S.}} \AND
  \bauthor{\bsnm{Jones},~\bfnm{Harold~L.}\binits{H.~L.}}
(\byear{1976}).
\btitle{A generalized likelihood ratio approach to the detection and estimation
  of jumps in linear systems}.
\bjournal{IEEE Transactions on Automatic Control}
\bvolume{21}
\bpages{108--112}.
\end{barticle}
\endbibitem

\bibitem{Xie&Siegmund-AS13}
\begin{barticle}[author]
\bauthor{\bsnm{Xie},~\bfnm{Yao}\binits{Y.}} \AND
  \bauthor{\bsnm{Siegmund},~\bfnm{David}\binits{D.}}
(\byear{2013}).
\btitle{Sequential multi-sensor change-point detection}.
\bjournal{Annals of Statistics}
\bvolume{41}
\bpages{670--692}.
\end{barticle}
\endbibitem

\end{thebibliography}

\end{document}